\documentclass[3p]{elsarticle}
\usepackage{amsmath,amsthm,,amsfonts,verbatim}
\newtheorem{theorem}{Theorem}[section]
\newtheorem{lemma}{Lemma}[section]

\newtheorem{remark}{Remark}[section]
\newtheorem{corollary}{Corollary}[section]

\usepackage[ruled]{algorithm2e}
\usepackage{algorithmic}
\usepackage{amsthm}
\usepackage{cases}
\usepackage{empheq}
\usepackage[figuresright]{rotating}
\usepackage{graphicx}
\usepackage{subfigure}
\usepackage{caption}
\usepackage{multicol,booktabs,,tabularx}
\usepackage{array}
\usepackage{amssymb,verbatim}

\usepackage{xcolor}
\usepackage[colorlinks=true,linkcolor=blue,citecolor=blue,urlcolor=blue]{hyperref}
\numberwithin{equation}{section} 
\begin{document}

\begin{frontmatter}
\title{SPICE: Scaling-Aware Prediction Correction Methods with a Free Convergence Rate for Nonlinear Convex Optimization}

\author{ Sai Wang$^{1,2}$} 
\address{$^1$ XXX University, Country\\
         $^2$Southern University of Science and Technology, China }

\begin{abstract}
Recently, the prediction-correction method has been developed to solve nonlinear convex optimization problems. However, its convergence rate is often poor since large regularization parameters are set to ensure convergence conditions. In this paper, the scaling-aware prediction correction (\textsf{Spice}) method is proposed to achieve a free convergence rate. This method adopts a novel scaling technique that adjusts the weight of the objective and constraint functions. The theoretical analysis demonstrates that increasing the scaling factor for the objective function or decreasing the scaling factor for constraint functions significantly enhances the convergence rate of the prediction correction method. In addition, the \textsf{Spice} method is further extended to solve separable variable nonlinear convex optimization. By employing different scaling factors as functions of the iterations, the \textsf{Spice} method achieves convergence rates of $\mathcal{O}(1/(t+1))$, $\mathcal{O}(1/[e^{t}(t+1)])$, and $\mathcal{O}(1/(t+1)^{t+1})$. Numerical experiments further validate the theoretical findings, demonstrating the effectiveness of the \textsf{Spice} method in practice.
\end{abstract}

\end{frontmatter}

\section{Introduction}
Nonlinear convex optimization is fundamental to many fields, including control theory, image denoising, and signal processing. This work addresses a general nonlinear convex problem with inequality constraints, formalized as:
\begin{equation}
\begin{aligned}
\mathsf{P}0:\quad \min\left\{ f(\mathbf{x}) \mid   \phi_{i}(\mathbf{x}) \leq 0, \ \mathbf{x}\in \mathcal{X},\ i=1,\cdots,p\right\},
\end{aligned}\label{sec1-eq1}
\end{equation}
where the set $\mathcal{X}$ is a nonempty closed convex subset of $\mathbb{R}^{n}$. The objective and constraint function $\{f(\mathbf{x}): \mathbb{R}^{n} \rightarrow \mathbb{R}\}$ and $\{\phi_{i}(\mathbf{x}): \mathbb{R}^{n} \rightarrow \mathbb{R}, i = 1, \cdots, p\}$ are convex, with the constraint functions also being continuously differentiable.  Notably, \( f(\mathbf{x}) \) is not assumed to be differentiable. The Lagrangian function associated with \textsf{P}0, parameterized by the dual variable \( \boldsymbol{\lambda} \), is given by:
\begin{equation}
\mathcal{L}(\mathbf{x}, \boldsymbol{\lambda}) = f(\mathbf{x}) + \boldsymbol{\lambda}^{\top} \Phi(\mathbf{x}),\ \label{sec1-eq2}
\end{equation}
where $\Phi(\mathbf{x}) = [\phi_{1}(\mathbf{x}), \cdots, \phi_{p}(\mathbf{x})]^{\top}$, and the dual variable $\boldsymbol{\lambda}$ belongs to the set $\mathcal{Z} := \mathbb{R}^{p}_{+}$. In this way, the constrained problem is reformulated to a saddle point problem. The saddle point of \eqref{sec1-eq2} yields the optimal solution of \textsf{P}0. A common approach is to design an iterative scheme to update the primal and dual variables alternately. The Arrow-Hurwicz method, one of the earliest primal-dual methods, was proposed to solve concave-convex optimization problems \cite{R0}, laying the foundation for subsequent developments. Sequentially, a primal-dual hybrid gradient (PDHG) method \cite{R3} introduced gradient steps, offering a more flexible framework. However, PDHG has been shown to diverge in certain linear programming problems \cite{R1}, and its variants achieve convergence only under more restrictive assumptions \cite{R1o2}. To address this issue, the Chambolle-Pock method was developed as a modified version of PDHG. By introducing a relaxation parameter, it improved stability and enhanced the convergence rate for saddle point problems with linear operators \cite{R6,R7,R8}. Further extensions of PDHG have been applied to nonlinear operators \cite{R4,R5}. Unfortunately, these approaches offer only local convergence guarantees, leaving the convergence rate unaddressed. Another notable class of primal-dual methods is the customized proximal point algorithms (PPA) developed in \cite{R9,R10,R11,R12,R13}. These algorithms aim to customize a symmetric proximal matrix, rather than merely relaxing the primal variable, to address convex problems with linear constraints. Based on variational analysis theory, they provide a straightforward proof of convergence. Later, a novel prediction-correction (PC) method was proposed in \cite{PC1,PC2}. The key innovation of the PC method is proposed to correct the primal and dual variables generated by the Arrow-Hurwicz method without constructing a symmetric proximal matrix. This unified framework establishes convergence conditions for the corrective matrix, which are relatively easy to satisfy. In Table \ref{table0}, a comparison of different primal-dual methods under setting assumptions and convergence rate.

Recent advances have extended prediction-correction (PC) methods to handle nonlinear optimization problems, particularly by leveraging carefully designed regularization parameters \cite{our1, our2}. These works prove that it can achieve ergodic convergence rates by ensuring that the sequence of regularization parameters is non-increasing. However, the design of such a sequence poses a significant challenge, as the regularization parameters strongly influence the performance of the proximal mapping problem. Larger regularization parameters, though useful in theory, often fail to capture the curvature of the proximal term effectively, leading to suboptimal convergence. While a smaller non-increasing sequence of regularization parameters can accelerate convergence toward the optimal solution, this approach introduces instability into the algorithm. The challenge, therefore, lies in balancing the size of these parameters to ensure both rapid convergence and algorithmic stability. The choice of these parameters is intricately tied to the derivative of the constraint functions; for any given derivative, a lower bound can be established on the product of the regularization parameters, which serves as a key condition for ensuring convergence. This lower bound prevents the parameters from being too small, while larger parameters are often inefficient at capturing the problem's local shape, thereby slowing convergence rates. Building upon these theoretical foundations, some works have introduced customized proximal matrices to optimize regularization parameter selection. For instance, the authors of \cite{PC3, our2} proposed generalized primal-dual methods that achieve a 25\% reduction in the lower bound of the product of regularization parameters, representing a significant improvement in parameter tuning. Despite this progress, identifying smaller regularization parameters that still meet convergence conditions remains a major obstacle.

To solve this challenge, this paper introduces a novel scaling technique that adjusts the weights of the objective and constraint functions, enabling a reduction in the regularization parameters. This innovation leads to the development of the scaling-aware prediction correction (\textsf{Spice}) method, which achieves a free convergence rate for nonlinear convex problems, overcoming the limitations of existing approaches in terms of both stability and efficiency.

\begin{table}[t]\centering
\caption{Convergence rates of prima-dual methods for convex optimizations ($t$ is the number of iterations, $\omega\in (0, 1)$).}
\scalebox{0.8}{
\begin{tabular}{@{}lcccc@{}}\toprule
\textbf{Algorithm} & \textbf{Linear Constraints} & \textbf{Nonlinear Constraints} & \textbf{Other Assumption} & \textbf{Convergence Rate} \\ \midrule
Arrow-Hurwicz \cite{R1o2}& $\surd$ &   $-$&  Strongly convex& $\mathcal{O}(1/t)$ \\\midrule
Chambbolle-Pock \cite{R7} & $\surd$ & $-$ & Smooth & $\mathcal{O}(\omega^{t})$ \\\midrule
Customized PPA \cite{R9,R10,R11}& $\surd$ &  $-$ & $-$ & $\mathcal{O}(1/t)$\\\midrule
Traditional PC \cite{PC1,PC2} & $\surd$ &  $-$ &  $-$& $\mathcal{O}(1/t)$\\ \midrule
\textsf{Spice} (this work)&  $-$ & $\surd$ & $-$& Free \\ \bottomrule
\end{tabular}}
\label{table0}
\end{table}

\section{Preliminaries}\label{section2}
This section introduces the foundational concepts, specifically the variational inequality, which is a key tool for the analysis later in the paper. Additionally, a scaling technique is proposed to adjust the variational inequality for further optimization.

\subsection{Variational Inequality}
This subsection formalizes the concept of variational inequality in the context of optimization problems. It provides the necessary mathematical background, including gradient representations and saddle point conditions, to establish the optimality criteria for constrained problems.
The gradient of the function $f(\mathbf{x})$ with respect to the vector $\mathbf{x}=[x_{1},\cdots,x_{n}]^{\top}$ is represented as $\mathcal{D}f(\mathbf{x})=[\nabla f(\mathbf{x})]^{\top}$. For a vector function $\Phi(\mathbf{x})$, its gradient is given by:
\begin{equation}
\mathcal{D} \Phi(\mathbf{x})=\left( \begin{array}{c} \mathcal{D}\phi_{1}(\mathbf{x}) \\ \vdots\\ \mathcal{D}\phi_{p}(\mathbf{x})\end{array} \right)=\left( \begin{array}{ccc} \frac{\partial \phi_{1}}{\partial x_{1}} &\cdots&\frac{\partial \phi_{1}}{\partial x_{n}} \\ \vdots&\ddots&\vdots\\ \frac{\partial \phi_{p}}{\partial x_{1}}  &\cdots&\frac{\partial \phi_{p}}{\partial x_{n}} \end{array} \right).\nonumber
\end{equation}
\begin{lemma}\label{sec2-lemma1}
Let $\mathcal{X}\subset \mathbb{R}^{n}$ be a closed convex set, with $f(\mathbf{x})$ and $h(\mathbf{x})$ as convex functions, where $h(\mathbf{x})$ is differentiable. Assume the minimization problem $\min\{ f(\mathbf{x})+ h(\mathbf{x})\mid \mathbf{x}\in \mathcal{X}\}$ has a nonempty solution set.  The vector $\mathbf{x}^{*}$ is an optimal solution, i.e.,
\begin{equation}
\mathbf{x}^{*}\in \arg \min \{ f(\mathbf{x})+  h(\mathbf{x})\mid\mathbf{x}\in \mathcal{X}\}, \nonumber
\end{equation}
if and only if 
\begin{equation}
\mathbf{x}^{*}\in \mathcal{X}, \quad f(\mathbf{x})- f(\mathbf{x}^{*})+  (\mathbf{x}-\mathbf{x}^{*})^{\top}\nabla  h(\mathbf{x}^{*})\ge0,\quad\forall \ \mathbf{x}\in \mathcal{X}.\nonumber
\end{equation}
\end{lemma}
\begin{proof}
Please refer to the work in \cite{our1} (Lemma 2.1).
\end{proof}

 Let $(\mathbf{x}^{*},\boldsymbol{\lambda}^{*})$ be the saddle point of \eqref{sec1-eq2} that satisfies the following conditions: 
\begin{empheq}[left=\empheqlbrace]{alignat=2}
&\mathbf{x}^{*}\in\arg\min\{\mathcal{L}(\mathbf{x},\boldsymbol{\lambda}^{*})\mid \mathbf{x}\in \mathcal{X}\},\nonumber\\[0.1cm] 
&\boldsymbol{\lambda}^{*}\in\arg\max\{\mathcal{L}(\mathbf{x}^{*},\boldsymbol{\lambda})\mid \boldsymbol{\lambda}\in \mathcal{Z}\}.\nonumber
\end{empheq}
By Lemma \ref{sec2-lemma1}, the saddle point follows the variational inequalities:
\begin{empheq}[left=\empheqlbrace]{alignat=2}
    & \mathbf{x}^{*}\in \mathcal{X},\quad&  f(\mathbf{x})- f(\mathbf{x}^{*})+(\mathbf{x}-\mathbf{x}^{*})^{\top}\mathcal{D} \Phi (\mathbf{x}^{*})^{\top}\boldsymbol{\lambda}^{*}\ge0,&\quad \forall \ \mathbf{x}\in \mathcal{X},\nonumber\\[0.1cm] 
    & \boldsymbol{\lambda}^{*}\in  \mathcal{Z},\quad& (\boldsymbol{\lambda}-\boldsymbol{\lambda}^{*})^{\top} [- \Phi (\mathbf{x}^{*})]\ge0,&\quad \forall \ \boldsymbol{\lambda}\in  \mathcal{Z},\nonumber
\end{empheq}
where $\mathcal{D} \Phi (\mathbf{x})=[\nabla \phi_{1},\cdots,\nabla \phi_{p}]^{\top}\in \mathbb{R}^{p\times n}$ and $\mathcal{D} \Phi (\mathbf{x})^{\top}\boldsymbol{\lambda}=\sum_{i=1}^{p} \lambda_{i}\nabla \phi_{i}(\mathbf{x})\in \mathbb{R}^{n\times 1}.$
The following monotone variational inequality can characterize the optimal condition:
\begin{equation}
 \mathbf{w}^{*}\in \Omega, \quad f(\mathbf{x})- f(\mathbf{x}^{*})+(\mathbf{w}-\mathbf{w}^{*})^{\top}\boldsymbol{\Gamma}(\mathbf{w}^{*})\ge0, \quad\forall \ \mathbf{w}\in \Omega, \label{sec2-eq1}
\end{equation}
where 
\begin{equation}
 \mathbf{w}=\left( \begin{array}{c} \mathbf{x} \\
      \boldsymbol{\lambda} \end{array} \right), \quad
 \boldsymbol{\Gamma}(\mathbf{w})=\left( \begin{array}{c} \mathcal{D} \Phi (\mathbf{x})^{\top}\boldsymbol{\lambda} \\
        -\Phi (\mathbf{x}) \end{array} \right) ,\quad\mbox{ and } \quad\Omega=\mathcal{X}\times \mathcal{Z} .\label{sec2-eq2}
\end{equation}

\begin{lemma}\label{sec2-lemma2}
Let $\mathcal{X}\subset \mathbb{R}^{n}$, $ \mathcal{Z}:=\mathbb{R}^{p_{1}}_{+}\times \mathbb{R}^{p_{2}}$ be closed convex sets. Then the operator $\boldsymbol{\Gamma}$ defined in (\ref{sec2-eq4}) satisfies 
\begin{equation}
(\mathbf{w}-\bar{\mathbf{w}})^{\top}\left[\boldsymbol{\Gamma}(\mathbf{w})-\boldsymbol{\Gamma}(\bar{\mathbf{w}})\right]\ge0, \quad\forall \ \mathbf{w}, \bar{\mathbf{w}}\in \Omega.\label{sec2-eq3}
\end{equation}
\end{lemma}
\begin{proof}
The proof can be found in \cite{our1} (Lemma 2.2).
\end{proof}

\subsection{Scaling Technique}
This subsection introduces a novel scaling technique to scale the objective and constraint functions. It shows how scaling factors impact the variational inequality, leading to a reformulated optimization problem with improved flexibility in solution methods.
\begin{lemma}\label{sec2-lemma3}
For any scaling factors $\rho>0,\eta>0$, if $\mathbf{x}^{*}$ is the optimal solution of  $\mathsf{P}0$, the following variational inequality holds
\begin{equation}
 \mathbf{w}^{*}\in \Omega, \quad \rho\left[f(\mathbf{x})- f(\mathbf{x}^{*})\right]+(\mathbf{w}-\mathbf{w}^{*})^{\top}\frac{1}{\eta}\boldsymbol{\Gamma}(\mathbf{w}^{*})\ge0, \quad\forall \ \mathbf{w}\in \Omega. \label{sec2-eq4}
\end{equation}
\end{lemma}
\begin{proof}
For any scaling factors  $\rho>0, \eta>0$, the original problem $\mathsf{P}0$ can be rewritten as 
\begin{equation}
\begin{aligned}
 \mathsf{P}1:\quad \min\left\{ \rho f(\mathbf{x}) \mid   \frac{1}{\eta}\phi_{i}(\mathbf{x}) \leq 0, \frac{1}{\eta}(\mathbf{A}\mathbf{x}-\mathbf{b})=\mathbf{0}, \ \mathbf{x}\in \mathcal{X},\ i=1,\cdots,p_{1}\right\}.
\end{aligned}\label{sec2-eq5}
\end{equation}
The scaling Lagrangian function of $\mathsf{P}1$ is defined as 
\begin{equation}
  \mathcal{L}(\mathbf{x},\boldsymbol{\lambda},\rho,\eta)= \rho f(\mathbf{x})+\boldsymbol{\lambda}^{\top} \frac{1}{\eta} \Phi (\mathbf{x}).\ \label{sec2-eq6}
\end{equation}
For given fixed values of $\rho$ and $\eta$, the saddle point of (\ref{sec2-eq6}) can written as
\begin{empheq}[left=\empheqlbrace]{alignat=2}
         &\mathbf{x}^{*}\in\arg\min\{\mathcal{L}(\mathbf{x},\boldsymbol{\lambda}^{*},\rho,\eta)\mid \mathbf{x}\in \mathcal{X}\},\nonumber\\[0.1cm] 
         &\boldsymbol{\lambda}^{*}\in\arg\max\{\mathcal{L}(\mathbf{x}^{*},\boldsymbol{\lambda},\rho,\eta)\mid \boldsymbol{\lambda}\in \mathcal{Z}\}.\nonumber
  \end{empheq}
They further follow the variational inequalities below:
\begin{empheq}[left=\empheqlbrace]{alignat=2}
     \mathbf{x}^{*}\in \mathcal{X},&& \quad \rho\left[f(\mathbf{x})- f(\mathbf{x}^{*})\right]+(\mathbf{x}-\mathbf{x}^{*})^{\top}\frac{1}{\eta}\mathcal{D} \Phi (\mathbf{x}^{*})^{\top}\boldsymbol{\lambda}^{*}\ge0,& \quad\forall \ \mathbf{x}\in \mathcal{X},\nonumber\\[0.1cm] 
     \boldsymbol{\lambda}^{*}\in  \mathcal{Z},&&\quad (\boldsymbol{\lambda}-\boldsymbol{\lambda}^{*})^{\top} [- \frac{1}{\eta}\Phi (\mathbf{x}^{*})]\ge0,&\quad \forall \ \boldsymbol{\lambda}\in  \mathcal{Z}.\nonumber
  \end{empheq}
The above inequalities can be described as a unified variational inequality:
\begin{equation}
 \mathbf{w}^{*}\in \Omega, \quad \rho\left[f(\mathbf{x})- f(\mathbf{x}^{*})\right]+(\mathbf{w}-\mathbf{w}^{*})^{\top}\frac{1}{\eta}\boldsymbol{\Gamma}(\mathbf{w}^{*})\ge0, \quad\forall \ \mathbf{w}\in \Omega. \nonumber
\end{equation}
Thus, this lemma is proven.
\end{proof}

\begin{remark}
For any $\rho>0$ and $\eta>0$, the variational inequality (\ref{sec2-eq4}) can be reformulated as
\begin{equation}
 \mathbf{w}^{*}\in \Omega, \quad f(\mathbf{x})- f(\mathbf{x}^{*})+(\mathbf{w}-\mathbf{w}^{*})^{\top}\frac{1}{\eta\rho}\boldsymbol{\Gamma}(\mathbf{w}^{*})\ge0, \quad\forall \ \mathbf{w}\in \Omega.\label{sec2-eq7}
\end{equation}
Since $f(\mathbf{x})- f(\mathbf{x}^{*})\ge0$ and $(\mathbf{w}-\mathbf{w}^{*})^{\top}\boldsymbol{\Gamma}(\mathbf{w}^{*})\ge0$, the inequality \eqref{sec2-eq7} implies that the combination of the function value difference and the term involving the gradient $\boldsymbol{\Gamma}$ always results in a non-negative value. This indicates that $\mathbf{w}^{*}$ is indeed an optimal solution in the feasible set $\Omega$ because any deviation $\mathbf{w}$ from $\mathbf{w}^{*}$ within the set $\Omega$ does not reduce the objective function value. Thus, the variational inequality effectively serves as a condition ensuring that $\mathbf{x}^{*}$ minimizes the function $f$ over the set $\Omega$.
\end{remark}

\section{Scaling-Aware Prediction Correction Method}
The \textsf{Spice} method consists of two operation steps. The first one involves utilizing the PPA scheme to predict the variables. In the second step, a matrix is applied to correct the predicted variables.

\begin{lemma}\label{sec3-lemma1}
Let $\mathbf{Q}$ be a second-order diagonal scalar upper triangular block matrix defined as:
\begin{equation}
\mathbf{Q} = \begin{pmatrix} 
r \mathbf{I}_n & \boldsymbol{\Lambda} \\ 
\mathbf{0} & s \mathbf{I}_p 
\end{pmatrix},\nonumber
\end{equation}
where $r>0, s >0$ are two positive constants, $\mathbf{I}_n$ and $\mathbf{I}_p$ are identity matrices, and $\boldsymbol{\Lambda} $ is an $n \times p$ matrix. For any vector $\mathbf{w} \in \mathbb{R}^{n+p}$, the following equality holds:
\begin{equation}
\mathbf{w}^\top \mathbf{Q} \mathbf{w} = \mathbf{w}^\top \frac{1}{2} (\mathbf{Q} + \mathbf{Q}^\top) \mathbf{w}.\nonumber
\end{equation}
\end{lemma}
\begin{proof}
Given the matrix $\mathbf{Q}$ and vector $\mathbf{w}$, the quadratic form is:
\[
\mathbf{w}^\top \mathbf{Q} \mathbf{w} = \begin{pmatrix} \mathbf{x}^\top & \boldsymbol{\lambda}^\top \end{pmatrix} \begin{pmatrix} r \mathbf{I}_n & \boldsymbol{\Lambda}  \\ \mathbf{0} & s \mathbf{I}_p \end{pmatrix} \begin{pmatrix} \mathbf{x} \\ \boldsymbol{\lambda} \end{pmatrix}.
\]
Expanding this expression:
\[
\mathbf{w}^\top \mathbf{Q} \mathbf{w} = r \|\mathbf{x}\|^2 + \mathbf{x}^\top \boldsymbol{\Lambda}  \boldsymbol{\lambda} + s \|\boldsymbol{\lambda}\|^2.
\]
Now, consider the symmetric part of $\mathbf{Q}$, which is $\frac{1}{2} (\mathbf{Q} + \mathbf{Q}^\top)$:
\[
\mathbf{Q}^\top = \begin{pmatrix} r \mathbf{I}_n & \mathbf{0} \\ \boldsymbol{\Lambda}^\top & s \mathbf{I}_p \end{pmatrix}.
\]
Thus:
\[
\frac{1}{2} (\mathbf{Q} + \mathbf{Q}^\top) = \frac{1}{2} \begin{pmatrix} r \mathbf{I}_n & \boldsymbol{\Lambda} \\ \mathbf{0} & s \mathbf{I}_p \end{pmatrix} + \frac{1}{2} \begin{pmatrix} r \mathbf{I}_n & \mathbf{0} \\ \boldsymbol{\Lambda} ^\top & s \mathbf{I}_p \end{pmatrix} = \begin{pmatrix} r \mathbf{I}_n & \frac{1}{2} \boldsymbol{\Lambda}  \\ \frac{1}{2} \ \boldsymbol{\Lambda} ^\top & s \mathbf{I}_p \end{pmatrix}.
\]
Now, the quadratic form $\mathbf{w}^\top \frac{1}{2} (\mathbf{Q} + \mathbf{Q}^\top) \mathbf{w}$ is
\[
\mathbf{w}^\top \frac{1}{2} (\mathbf{Q} + \mathbf{Q}^\top) \mathbf{w} = \begin{pmatrix} \mathbf{x}^\top & \boldsymbol{\lambda}^\top \end{pmatrix} \begin{pmatrix} r \mathbf{I}_n & \frac{1}{2} \boldsymbol{\Lambda}  \\ \frac{1}{2} \boldsymbol{\Lambda}^\top & s \mathbf{I}_p \end{pmatrix} \begin{pmatrix} \mathbf{x} \\ \boldsymbol{\lambda} \end{pmatrix}.
\]
Expanding this:
\[
\mathbf{w}^\top \frac{1}{2} (\mathbf{Q} + \mathbf{Q}^\top) \mathbf{w} = r \|\mathbf{x}\|^2 + \frac{1}{2} \mathbf{x}^\top \boldsymbol{\Lambda}  \boldsymbol{\lambda} + \frac{1}{2} \boldsymbol{\lambda}^\top \boldsymbol{\Lambda} ^\top \mathbf{x} + s \|\boldsymbol{\lambda}\|^2.
\]
Since $\mathbf{x}^\top \boldsymbol{\Lambda}  \boldsymbol{\lambda}$ and $\boldsymbol{\lambda}^\top \boldsymbol{\Lambda} ^\top \mathbf{x}$ are scalars and equal, we have
\[
\mathbf{w}^\top \frac{1}{2} (\mathbf{Q} + \mathbf{Q}^\top) \mathbf{w} = r \|\mathbf{x}\|^2 + s \|\boldsymbol{\lambda}\|^2 + \mathbf{x}^\top \boldsymbol{\Lambda} \boldsymbol{\lambda}.
\]
This result matches the original quadratic form:
\[
\mathbf{w}^\top \mathbf{Q} \mathbf{w} = r \|\mathbf{x}\|^2 + s \|\boldsymbol{\lambda}\|^2 + \mathbf{x}^\top \boldsymbol{\Lambda} \boldsymbol{\lambda}.
\]
Thus, this lemma holds.
\end{proof}
\begin{remark}\label{sec3-remark1}
This corrected proof demonstrates that for any second-order diagonal scalar upper triangular block matrix, the quadratic form $\mathbf{w}^\top \mathbf{Q} \mathbf{w}$ is indeed equal to the quadratic form $\mathbf{w}^\top \frac{1}{2} (\mathbf{Q} + \mathbf{Q}^\top) \mathbf{w}$, without any additional factors in the mixed term. In particular, if $r\cdot s> \frac{1}{4}\| \boldsymbol{\Lambda} \|^{2}$ holds, we have $\frac{1}{2} (\mathbf{Q} + \mathbf{Q}^\top)\succ 0$ and $\mathbf{w}^\top \mathbf{Q} \mathbf{w} = \mathbf{w}^\top \frac{1}{2} (\mathbf{Q} + \mathbf{Q}^\top) \mathbf{w}=\|\mathbf{w}\|_{\frac{1}{2} (\mathbf{Q} + \mathbf{Q}^\top)}>0.$
\end{remark}

\subsection{PPA-Like Prediction Scheme}
The PPA highlights its performance between adjacent iterative steps. For the $k$-th iteration and regularization parameters $r_{k}>0,s_{k}>0$, the PPA-like prediction scheme can be described by the following equations:
\begin{subequations}\label{sec3-eq1}
\begin{empheq}[left=\empheqlbrace]{alignat=2} 
&\bar{\mathbf{x}}^{k} = \arg\min \left\{ \mathcal{ L}(\mathbf{x},\boldsymbol{\lambda}^{k},\rho,\eta_k) + \frac{r_{k}}{2}\|\mathbf{x} - \mathbf{x}^{k}\|^2 \mid \mathbf{x}\in \mathcal{X} \right\},\\[0.1cm]
&\boldsymbol{\bar{\lambda}}^{k} = \arg\max \left\{ \mathcal{ L}(\bar{\mathbf{x}}^{k}, \boldsymbol{\lambda},\rho,\eta_k) - \frac{s_{k}}{2}\|\boldsymbol{\lambda} - \boldsymbol{\lambda}^{k}\|^2 \mid  \boldsymbol{\lambda}\in \mathcal{Z} \right\},
\end{empheq}
\end{subequations}
Referring to Lemma \ref{sec2-lemma1}, the saddle point of ($\bar{\mathbf{x}}^{k} ,\boldsymbol{\bar{\lambda}}^{k} $) satisfies the following variational inequalities:
\begin{empheq}[left=\empheqlbrace]{alignat=2}
\bar{\mathbf{x}}^{k}\in \mathcal{X}, &&\quad \rho \left[f(\mathbf{x})- f(\bar{\mathbf{x}}^{k})\right]+(\mathbf{x}-\bar{\mathbf{x}}^{k})^{\top}\left[\frac{1}{\eta_k}\mathcal{D} \Phi (\bar{\mathbf{x}}^{k})^{\top}\boldsymbol{\bar{\lambda}}^{k}\right.\qquad&\nonumber \\[0.1cm]
&&\left.-\frac{1}{\eta_k}\mathcal{D} \Phi (\bar{\mathbf{x}}^{k})^{\top}(\boldsymbol{\bar{\lambda}}^{k}-\boldsymbol{\bar{\lambda}}^{k})+r_{k}(\bar{\mathbf{x}}^{k}-\mathbf{x}^k)\right]& \ge 0, \quad \forall \ \mathbf{x}\in \mathcal{X}, \label{eq22a}\nonumber  \\[0.1cm]
\boldsymbol{\bar{\lambda}}^{k}\in  \mathcal{Z}, &&\quad(\boldsymbol{\lambda}-\boldsymbol{\bar{\lambda}}^{k})^{\top} \left[- \frac{1}{\eta_k}\Phi (\bar{\mathbf{x}}^{k})+ s_{k}(\boldsymbol{\bar{\lambda}}^{k}-\boldsymbol{\lambda}^{k})\right] &\ge 0, \quad \forall \ \boldsymbol{\lambda}\in  \mathcal{Z}. \nonumber 
\end{empheq}
This can be rewritten in a unified form as:
\begin{equation}
  \bar{\mathbf{w}}^{k}\in \Omega, \quad \rho \left[ f(\mathbf{x}) - f(\bar{\mathbf{x}}^{k}) \right] + (\mathbf{w} - \bar{\mathbf{w}}^{k})^{\top} \left[ \frac{1}{\eta_k} \boldsymbol{\Gamma}(\bar{\mathbf{w}}^{k}) + \mathbf{Q}_{k} (\bar{\mathbf{w}}^{k} - \mathbf{w}^{k}) \right] \geq 0, \quad \forall \ \mathbf{w}\in \Omega, \label{sec3-eq2}
\end{equation}
where the predictive (proximal) matrix $\mathbf{Q}_{k}$ is given as:
\begin{equation}
\mathbf{Q}_{k} = \left( \begin{array}{cc} r_{k} \mathbf{I}_{n} & -\frac{1}{\eta_k}\mathcal{D} \Phi (\bar{\mathbf{x}}^{k})^{\top}\\[0.1cm]
        \mathbf{0} & s_{k} \mathbf{I}_{p} \end{array} \right).\nonumber 
\end{equation}
In the PPA-like prediction scheme, the proximal term $\mathbf{Q}_{k}$ is a second-order diagonal scalar upper triangular block matrix. By Remark \ref{sec3-remark1}, if $\mathbf{w} = \mathbf{w}^{k}$ and $r_{k}\cdot s_{k}> \frac{1}{4\eta_{k}^{2}}\|\mathcal{D} \Phi (\bar{\mathbf{x}}^{k})\|^{2}$ hold, the inequality \eqref{sec3-eq2} can be simplified as:
\begin{equation}
  \rho \left[ f(\mathbf{x}^{k}) - f(\bar{\mathbf{x}}^{k}) \right] + (\mathbf{w}^{k} - \bar{\mathbf{w}}^{k})^{\top} \frac{1}{\eta_k} \boldsymbol{\Gamma}(\bar{\mathbf{w}}^{k}) \geq (\mathbf{w}^{k}-\bar{\mathbf{w}}^{k} )^{\top}{\mathbf{Q}_{k}}(\mathbf{w}^{k}-\bar{\mathbf{w}}^{k})>0. \nonumber
\end{equation}
Although the generated sequence satisfies the variational condition, it may converge to a suboptimal solution. 
It has been shown in \cite{our1} that a symmetrical proximal matrix leads to the convergence of the customized PPA. Since the proximal matrix $\mathbf{Q}_{k}$ is not symmetrical, the PPA diverges for some simple linear convex problems \cite{R1}. To ensure the convergence of the proposed method, the predictive variables should be corrected by a corrective matrix. The pseudocode for the \textsf{Spice} method can be found in Algorithm \ref{alg1}.

 \begin{algorithm}[t]
\caption{Scaling-aware prediction correction method}\label{alg1}
\LinesNumbered
\KwIn{ Parameters $\mu>1$; tolerant error $\tau$.}
\KwOut{The optimal solution: $\mathbf{x}^{*}$.}
Initialize $\mathbf{x}^{0},\boldsymbol{\lambda}^{0},k=0,\eta_{0}=1$\;
\While{$error\ge\tau\ $}{
\emph{\% The prediction step:}\\
\If{$k>0$}{
$\eta_{k}'=\eta_{k-1}$; $\eta_{k}=\eta_{k}'$; $\eta_{max}=\mu\cdot\eta_{k}$\; 
\While{ $\eta_{k}< \eta_{max}$}{
$\eta_{k} = \eta_{k}'$; $r_{k}=\frac{1}{\eta_{k}}\sqrt{\mathsf{R}(\mathbf{x}^{k})}$\;
$\bar{\mathbf{x}}^{k}=\arg\min\{\mathcal{L}(\mathbf{x},\boldsymbol{\lambda}^{k},\rho,\eta_{k})+\frac{r_{k}}{2}\|\mathbf{x}-\mathbf{x}^{k}\|^{2}\mid \mathbf{x}\in \mathcal{X}\}$\;
$\eta_{max} = \max \left\{\eta_{k-1} \cdot \sqrt{\frac{\mathsf{R}(\mathbf{x}^{k})}{\mathsf{R}(\mathbf{x}^{k-1})}},\quad \eta_{k-1} \cdot \frac{\mathsf{R}(\bar{\mathbf{x}}^{k}) \sqrt{\mathsf{R}(\mathbf{x}^{k-1})}}{\mathsf{R}(\bar{\mathbf{x}}^{k-1}) \sqrt{\mathsf{R}(\mathbf{x}^{k})}} \right\}$\;
$\eta_{k}'=\mu \cdot \eta_{k}$\;
}}
$s_{k}=\frac{\mu}{\eta_{k}} \mathsf{R}(\bar{\mathbf{x}}^{k})/\sqrt{\mathsf{R}(\mathbf{x}^{k})}$\;
$\bar{\boldsymbol{\lambda}}^{k}=\mathsf{P}_{\mathcal{Z}}\left(\boldsymbol{\lambda}^{k}+\frac{1}{\eta_{k}s_{k}}  \Phi (\bar{\mathbf{x}}^{k}) \right)$\;
\emph{\% The correction step:}\\
$\mathbf{w}^{k+1}=\mathbf{w}^{k}-\mathbf{M}_{k}(\mathbf{w}^{k}-\bar{\mathbf{w}}^{k})$\;
$error=\mbox{abs}[ f(\mathbf{x}^{k})- f(\mathbf{x}^{k+1})]$\;
$k=k+1$\;
} 
\Return { $\mathbf{x}^{*}=\mathbf{x}^{k}$}\;
 \end{algorithm}

\subsection{Matrix-Driven Correction Scheme}

In the correction phase, a corrective matrix is employed to adjust the predictive variables. It is important to note that this matrix is not unique; various forms can be utilized, including both upper and lower triangular matrices, as discussed in related work \cite{our1}.

\begin{equation}\label{sec3-eq3}
\mathbf{w}^{k+1} =\mathbf{w}^{k}- \mathbf{M}_{k}(\mathbf{w}^{k}-\bar{\mathbf{w}}^{k}).
\end{equation}
For the sake of simplifying the analysis, in this paper, the corrective matrix $\mathbf{M}_{k}$ is defined as follows:
\begin{equation}
\mathbf{M}_{k}=\left( \begin{array}{cc}\mathbf{I}_{n} &-\frac{1}{\eta_k r_{k}}\mathcal{D} \Phi (\bar{\mathbf{x}}^{k})^{\top}\\[0.1cm]
        \mathbf{0} & \mathbf{I}_{p}\end{array} \right).\nonumber 
\end{equation}
Referring to equation (\ref{sec3-eq3}), the corrective matrix is intrinsically linked to the predictive variable and the newly iterated variables. To facilitate convergence analysis, two extended matrices, $\mathbf{H}_{k}$ and $\mathbf{G}_{k}$, are introduced. These matrices are derived by dividing by $\mathbf{Q}_{k}$, setting the groundwork for the subsequent section. 
\begin{eqnarray} 
  \mathbf{H}_{k}  =  \mathbf{Q}_{k}\mathbf{M}_{k}^{-1}
       & = & \left( \begin{array}{cc}r_{k}\mathbf{I}_{n} &-\frac{1}{\eta_k}\mathcal{D} \Phi (\bar{\mathbf{x}}^{k})^{\top}\\[0.1cm]
        \mathbf{0} & s_{k}\mathbf{I}_{p}\end{array} \right)  \left( \begin{array}{cc}\mathbf{I}_{n} &\frac{1}{\eta_k r_{k}}\mathcal{D} \Phi (\bar{\mathbf{x}}^{k})^{\top}\\[0.1cm]
        \mathbf{0} & \mathbf{I}_{p}\end{array} \right) \nonumber \\
   & = & \left(\begin{array}{cc}
         r_{k}\mathbf{I}_{n} &  \mathbf{0}\\[0.1cm]
         \mathbf{0}  & s_{k}\mathbf{I}_{p} \end{array} \right),\nonumber 
\end{eqnarray}
The first extended matrix, $\mathbf{H}_{k}$, is diagonal. When $r_{k}\approx s_{k}$, it can be considered as a scaled identity matrix by $r_{k}$. Clearly, it is a positive definite matrix since both $r_{k}$ and $s_{k}$ are positive. The second extended matrix, $\mathbf{G}_{k}$, is defined as follows:

\begin{eqnarray}  \label{Matrix-G}
  \mathbf{G}_{k}  &=  &    \mathbf{Q}_{k}^{\top} +\mathbf{Q}_{k} -  \mathbf{M}_{k}^{\top}\mathbf{H}_{k}\mathbf{M}_{k}   =   \mathbf{Q}_{k}^{\top} +\mathbf{Q}_{k} -  \mathbf{M}_{k}^{\top}\mathbf{Q}_{k}       \nonumber \\[0.1cm]
       & = &  \left(\begin{array}{cc}
         2r_{k}\mathbf{I}_{n} &   -\frac{1}{\eta_k}\mathcal{D} \Phi (\bar{\mathbf{x}}^{k})^{\top}\\[0.1cm]
         -\frac{1}{\eta_k}\mathcal{D} \Phi (\bar{\mathbf{x}}^{k})   & 2s_{k}\mathbf{I}_{p} \end{array} \right) -  \left( \begin{array}{cc}\mathbf{I}_{n} &\mathbf{0}\\[0.1cm]
        -\frac{1}{\eta_k r_{k}}\mathcal{D} \Phi (\bar{\mathbf{x}}^{k})& \mathbf{I}_{p}\end{array} \right) \left( \begin{array}{cc}r_{k}\mathbf{I}_{n} &-\frac{1}{\eta_k}\mathcal{D} \Phi (\bar{\mathbf{x}}^{k})^{\top}\\[0.1cm]
        \mathbf{0} & s_{k}\mathbf{I}_{p}\end{array} \right)   \nonumber \\[0.1cm]
        & = &     \left(\begin{array}{cc}
         2r_{k}\mathbf{I}_{n} &   -\frac{1}{\eta_k}\mathcal{D} \Phi (\bar{\mathbf{x}}^{k})^{\top}\\[0.1cm]
         -\frac{1}{\eta_k}\mathcal{D} \Phi (\bar{\mathbf{x}}^{k})   & 2s_{k}\mathbf{I}_{p} \end{array} \right)  -
          \left(\begin{array}{cc}
       r_{k}\mathbf{I}_{n} &   -\frac{1}{\eta_k}\mathcal{D} \Phi (\bar{\mathbf{x}}^{k})^{\top}\\[0.1cm]
        -\frac{1}{\eta_k}\mathcal{D} \Phi (\bar{\mathbf{x}}^{k})  & s_{k}\mathbf{I}_{p}+\frac{1}{\eta_k^{2}r_{k}}\mathcal{D} \Phi (\bar{\mathbf{x}}^{k})\mathcal{D} \Phi (\bar{\mathbf{x}}^{k})^{\top}  \end{array} \right)    \nonumber  \\
          & =&  \left(\begin{array}{cc}
         r_{k}\mathbf{I}_{n} &  \mathbf{0}    \\[0.1cm]
         \mathbf{0}   & s_{k}\mathbf{I}_{p}-\frac{1}{\eta_k^{2}r_{k}}\mathcal{D} \Phi (\bar{\mathbf{x}}^{k})\mathcal{D} \Phi (\bar{\mathbf{x}}^{k})^{\top}  \end{array} \right).\nonumber 
       \end{eqnarray}
This results in a block diagonal matrix. By applying the Schur complement lemma, it is evident that if the condition $ r_{k}\cdot s_{k}\cdot \mathbf{I}_{n}\succ \frac{1}{\eta_k^{2}}\mathcal{D} \Phi (\bar{\mathbf{x}}^{k})^{\top}\mathcal{D} \Phi (\bar{\mathbf{x}}^{k}) $ is met, $\mathbf{G}_{k}$ is positive definite. In this paper, for $\mu>1$ and $k\ge1$, we define $\mathsf{R}(\mathbf{x}^{k})=\|\mathcal{D} \Phi (\mathbf{x}^{k})\|^{2}$ and set:
\begin{subequations}\label{sec3-eq4}
\begin{empheq}[left=\empheqlbrace]{alignat=2}
r_{k}&=\frac{1}{\eta_k}\sqrt{\mathsf{R}(\mathbf{x}^{k})},\\[0.1cm]
s_{k}&=\frac{\mu\mathsf{R}(\bar{\mathbf{x}}^{k})}{\eta_k\sqrt{\mathsf{R}(\mathbf{x}^{k})}}.
\end{empheq}
\end{subequations}
To ensure that the sequence $\eta_k$ satisfies the conditions $r_{k-1} \ge r_{k}$ and $s_{k-1} \ge s_k$, the sequence of $\{r_{k}, s_{k}\}$ is non-increasing. Starting with the condition $r_{k-1} \ge r_{k}$, we have
\begin{equation}
\frac{1}{\eta_{k-1}}\sqrt{\mathsf{R}(\mathbf{x}^{k-1})} \ge \frac{1}{\eta_k}\sqrt{\mathsf{R}(\mathbf{x}^{k})}.\nonumber
\end{equation}
Thus, the value of $\eta_k$ must satisfy:
\begin{equation}
\eta_k \ge \eta_{k-1} \cdot \sqrt{\frac{\mathsf{R}(\mathbf{x}^{k})}{\mathsf{R}(\mathbf{x}^{k-1})}}.\nonumber
\end{equation}
Similarly, beginning with the condition $s_{k-1} \ge s_k$, we have
\begin{equation}
\frac{\mu\mathsf{R}(\bar{\mathbf{x}}^{k-1})}{\eta_{k-1}\sqrt{\mathsf{R}(\mathbf{x}^{k-1})}} \ge \frac{\mu\mathsf{R}(\bar{\mathbf{x}}^{k})}{\eta_k\sqrt{\mathsf{R}(\mathbf{x}^{k})}}.\nonumber
\end{equation}
Therefore, $\eta_k$ must also satisfy:
\begin{equation}
\eta_k \ge \eta_{k-1} \cdot \frac{\mathsf{R}(\bar{\mathbf{x}}^{k}) \sqrt{\mathsf{R}(\mathbf{x}^{k-1})}}{\mathsf{R}(\bar{\mathbf{x}}^{k-1}) \sqrt{\mathsf{R}(\mathbf{x}^{k})}}.\nonumber
\end{equation}
To satisfy both $r_{k-1} \ge r_{k}$ and $s_{k-1} \ge s_k$, the parameter $\eta_k$ must be chosen to satisfy the more restrictive of the two conditions. Therefore, the appropriate choice for $\eta_k$ is given by:
\begin{equation}
\eta_k \ge \max \left\{\eta_{k-1} \cdot \sqrt{\frac{\mathsf{R}(\mathbf{x}^{k})}{\mathsf{R}(\mathbf{x}^{k-1})}},\quad \eta_{k-1} \cdot \frac{\mathsf{R}(\bar{\mathbf{x}}^{k}) \sqrt{\mathsf{R}(\mathbf{x}^{k-1})}}{\mathsf{R}(\bar{\mathbf{x}}^{k-1}) \sqrt{\mathsf{R}(\mathbf{x}^{k})}} \right\}.\label{sec3-add5}
\end{equation}
For any iteration $k\ge0$, we define a difference matrix as follows 
\begin{eqnarray}
\mathbf{D}_{k}=\mathbf{H}_{k}-\mathbf{H}_{k+1}=\left(\begin{array}{cc}
         (r_{k}-r_{k+1})\mathbf{I}_{n} &  \mathbf{0}\\[0.1cm]
         \mathbf{0}  & (s_{k}-s_{k+1})\mathbf{I}_{p} \end{array} \right).\nonumber 
\end{eqnarray}
Since $r_{k} \ge r_{k+1}$ and $s_{k} \ge s_{k+1}$ hold, $\mathbf{D}_{k}$ is a semi positive definite matrix. The initial extended matrix can be reformulated as $\mathbf{H}_{0} = \frac{1}{\eta_0}\mathsf{H}_{0}$, given by:
\begin{eqnarray} 
 \mathbf{H}_{0}  =  \left(\begin{array}{cc}
       \frac{1}{\eta_0}\sqrt{\mathsf{R}(\mathbf{x}^{0})}\mathbf{I}_{n} &   \mathbf{0}\\[0.1cm]
         \mathbf{0}   &\frac{\mu\mathsf{R}(\bar{\mathbf{x}}^{0})}{\eta_0\sqrt{\mathsf{R}(\mathbf{x}^{0})}} \mathbf{I}_{p} \end{array} \right),\quad
  \mathsf{H}_{0}  =  \left(\begin{array}{cc}
       \sqrt{\mathsf{R}(\mathbf{x}^{0})}\mathbf{I}_{n} &   \mathbf{0}\\[0.1cm]
         \mathbf{0}   &\frac{\mu\mathsf{R}(\bar{\mathbf{x}}^{0})}{\sqrt{\mathsf{R}(\mathbf{x}^{0})}} \mathbf{I}_{p} \end{array} \right).\nonumber 
\end{eqnarray}
From the structure of these matrices, it is evident that $\mathsf{H}_{0}$ is independent of the scaling factor $\eta_k$, relying solely on the initial value of $\mathsf{R}(\mathbf{x}^{0})$.

\subsection{Sequence Convergence Analysis}
Based on the previous preparation work, this part demonstrates the sequence convergence of the proposed method. First, the convergence condition is derived, followed by an analysis of the sequence convergence characteristics.
\begin{lemma}\label{sec3-lemma2}
Let $\{\mathbf{w} ^{k},\bar{\mathbf{w} }^{k},\mathbf{w} ^{k+1}\}$ be the sequences generated by the \textsf{Spice} method. When we use equation (\ref{sec3-eq4}) to update the values of $r_{k}$ and $s_{k}$, the predictive matrix $\mathbf{Q}_{k}$ and corrective matrix $\mathbf{M}_{k}$ are upper triangular and satisfy 
\begin{equation}
\mathbf{H}_{k}=\mathbf{Q}_{k}\mathbf{M}_{k}^{-1}\succ0 \quad {\rm and}\quad  \mathbf{G}_{k}=\mathbf{Q}_{k}^{\top}+\mathbf{Q}_{k}- \mathbf{M}_{k}^{\top}\mathbf{H}_{k}\mathbf{M}_{k}\succ0,\label{sec3-eq5}
\end{equation}
 and the following variational inequality holds
\begin{equation}\begin{aligned}
  \rho \left[f(\mathbf{x})- f(\bar{\mathbf{x}}^{k})\right]+(\mathbf{w} -\bar{\mathbf{w} }^{k})^{\top}\frac{1}{\eta_k}\boldsymbol{\Gamma}(\bar{\mathbf{w} }^{k})\ge\frac{1}{2}\|\bar{\mathbf{w}}^{k}-\mathbf{w} ^{k}\|_{\mathbf{G}_{k}}^{2}+\frac{1}{2 }\left(\|\mathbf{w} -\mathbf{w} ^{k+1}\|_{\mathbf{H}_{k}}^{2}-\|\mathbf{w} -\mathbf{w}^{k}\|_{\mathbf{H}_{k}}^{2}\right),\quad \forall \ \mathbf{w} \in& \Omega. \end{aligned}\label{sec3-eq6}
\end{equation}
\end{lemma}
\begin{proof}
Referring to \eqref{sec3-eq2} and \eqref{sec3-eq5}, the proximal term can be  written as
\begin{equation}\begin{aligned}
(\mathbf{w} -\bar{\mathbf{w} }^{k})^{\top}\mathbf{Q}_{k}(\mathbf{w} ^{k}-\bar{\mathbf{w} }^{k})=(\mathbf{w} -\bar{\mathbf{w} }^{k})^{\top}\mathbf{H}_{k}\mathbf{M}_{k}(\mathbf{w} ^{k}-\bar{\mathbf{w} }^{k})=(\mathbf{w} -\bar{\mathbf{w} }^{k})^{\top}\mathbf{H}_{k}(\mathbf{w} ^{k}-\mathbf{w} ^{k+1}).\label{sec3-eq7}
\end{aligned}\end{equation}
To further simplify it, the following fact can be used.
\begin{equation}
(\mathbf{w}-\mathbf{w}_1)^{\top}\mathbf{H}_{k}(\mathbf{w}_2-\mathbf{w}_{3})=\frac{1}{2}\left( \Vert\mathbf{w}-\mathbf{w}_{3}\Vert_{\mathbf{H}_{k}}^{2}-\Vert\mathbf{w}-\mathbf{w}_2\Vert_{\mathbf{H}_{k}}^{2}\right)+\frac{1}{2}\left( \Vert\mathbf{w}_1-\mathbf{w}_2\Vert_{\mathbf{H}_{k}}^{2}-\Vert\mathbf{w}_1-\mathbf{w}_{3}\Vert_{\mathbf{H}_{k}}^{2}\right).\label{sec3-eq8}
\end{equation}
Applying (\ref{sec3-eq8}) to the equation (\ref{sec3-eq7}), we obtain
\begin{equation}
\begin{aligned}
(\mathbf{w} -\bar{\mathbf{w} }^{k})^{\top}\mathbf{H}_{k}(\mathbf{w} ^{k}-\mathbf{w} ^{k+1})=&\frac{1}{2}\left( \|\mathbf{w} -\mathbf{w} ^{k+1}\|_{\mathbf{H}_{k}}^{2}-\|\mathbf{w} -\mathbf{w} ^{k}\|_{\mathbf{H}_{k}}^{2}\right)+\frac{1}{2}\left( \|\bar{\mathbf{w}}^{k}-\mathbf{w} ^{k}\|_{\mathbf{H}_{k}}^{2}-\|\bar{\mathbf{w} }^{k}-\mathbf{w} ^{k+1}\|_{\mathbf{H}_{k}}^{2}\right).\nonumber
\end{aligned}
\end{equation}
On the other hand, by Lemma \ref{sec3-lemma1}, the following equation holds.
\begin{eqnarray}
\lefteqn{\|\bar{\mathbf{w} }^{k}-\mathbf{w} ^{k}\|_{\mathbf{H}_{k}}^{2}-\|\bar{\mathbf{w} }^{k}-\mathbf{w} ^{k+1}\|_{\mathbf{H}_{k}}^{2}}\nonumber\\
&=&\|\bar{\mathbf{w} }^{k}-\mathbf{w}^{k}\|_{\mathbf{H}_{k}}^{2}-\|(\bar{\mathbf{w} }^{k}-\mathbf{w} ^{k})-(\mathbf{w} ^{k+1}-\mathbf{w} ^{k})\|_{\mathbf{H}_{k}}^{2}\nonumber\\
&=&\|\bar{\mathbf{w} }^{k}-\mathbf{w} ^{k}\|_{\mathbf{H}_{k}}^{2}-\|(\bar{\mathbf{w} }^{k}-\mathbf{w} ^{k})- \mathbf{M}_{k}(\bar{\mathbf{w} }^{k}-\mathbf{w} ^{k})\|_{\mathbf{H}_{k}}^{2}\nonumber\\
&=&2 (\bar{\mathbf{w} }^{k}-\mathbf{w} ^{k})^{\top}\mathbf{H}_{k}\mathbf{M}_{k}(\bar{\mathbf{w} }^{k}-\mathbf{w} ^{k})- (\bar{\mathbf{w} }^{k}-\mathbf{w} ^{k})^{\top}\mathbf{M}_{k}^{\top}\mathbf{H}_{k}\mathbf{M}_{k}(\bar{\mathbf{w} }^{k}-\mathbf{w} ^{k})\nonumber\\
&=& (\bar{\mathbf{w} }^{k}-\mathbf{w} ^{k})^{\top}(\mathbf{Q}_{k}^{\top}+\mathbf{Q}_{k}- \mathbf{M}_{k}^{\top}\mathbf{H}_{k}\mathbf{M}_{k})(\bar{\mathbf{w} }^{k}-\mathbf{w} ^{k})\nonumber\\
&=& \|\bar{\mathbf{w} }^{k}-\mathbf{w} ^{k}\|_{\mathbf{G}_{k}}^{2}.\nonumber
\end{eqnarray}
Combining with (\ref{sec3-eq7}), the proximal term is equal to 
\begin{equation}
(\mathbf{w} -\bar{\mathbf{w} }^{k})^{\top}\mathbf{Q}_{k}(\mathbf{w} ^{k}-\bar{\mathbf{w} }^{k})=\frac{1}{2 }\left( \|\mathbf{w} -\mathbf{w} ^{k+1}\|_{\mathbf{H}_{k}}^{2}-\|\mathbf{w} -\mathbf{w} ^{k}\|_{\mathbf{H}_{k}}^{2}\right)+\frac{1}{2}\|\bar{\mathbf{w} }^{k}-\mathbf{w} ^{k}\|_{\mathbf{G}_{k}}^{2},\label{sec3-eq9}
\end{equation}
By replacing the prediction term of (\ref{sec3-eq2}) with (\ref{sec3-eq9}), this lemma is proved.
\end{proof}

\begin{theorem}\label{sec3-theorem1}
Let $\{\mathbf{w} ^{k},\bar{\mathbf{w} }^{k},\mathbf{w} ^{k+1}\}$ be the sequences generated by the \textsf{Spice} method. For the predictive matrix $\mathbf{Q}_{k}$, if there is a corrective matrix $\mathbf{M}_{k}$ that satisfies the convergence condition (\ref{sec3-eq5}), then the sequences satisfy the following inequality
\begin{equation}
\Vert\mathbf{w} ^{*}-\mathbf{w} ^{k}\Vert_{\mathbf{H}_{k}}^{2}\ge \Vert\mathbf{w} ^{*}-\mathbf{w} ^{k+1}\Vert_{\mathbf{H}_{k}}^{2}+ \Vert\mathbf{w} ^{k}-\bar{\mathbf{w} }^{k}\Vert_{\mathbf{G}_{k}}^{2}, \quad\ \mathbf{w} ^{*}\in \Omega^{*}, \label{sec3-eq10}
\end{equation}
where $\Omega^{*}$ is the set of optimal solutions.
\end{theorem}
\begin{proof} Setting $\mathbf{w} =\mathbf{w} ^{*}$, the inequality of \eqref{sec3-eq6} can be reformulated as
\begin{equation}\begin{aligned}
\|\mathbf{w} ^{*}-\mathbf{w} ^{k}\|_{\mathbf{H}_{k}}^{2}-\|\mathbf{w} ^{*}-\mathbf{w} ^{k+1}\|_{\mathbf{H}_{k}}^{2}&- \Vert\mathbf{w} ^{k}-\bar{\mathbf{w} }^{k}\Vert_{\mathbf{G}_{k}}^{2}\ge 2 \left \{\rho [ f(\bar{\mathbf{x}}^{k})- f(\mathbf{x}^{*})]+(\bar{\mathbf{w} }^{k}-\mathbf{w} ^{*})^{\top}\frac{1}{\eta_k}\boldsymbol{\Gamma}(\bar{\mathbf{w} }^{k})\right\}.\nonumber
\end{aligned}\end{equation}
By Lemme \ref{sec2-lemma2}, the monotone operator satisfies $(\bar{\mathbf{w} }^{k}-\mathbf{w} ^{*})^{\top}\boldsymbol{\Gamma}(\bar{\mathbf{w} }^{k})\ge (\bar{\mathbf{w} }^{k}-\mathbf{w} ^{*})^{\top}\boldsymbol{\Gamma}(\mathbf{w} ^{*}).$ Then we  have
\begin{equation}\begin{aligned}
\Vert\mathbf{w} ^{*}-\mathbf{w} ^{k}\Vert_{\mathbf{H}_{k}}^{2}-\Vert\mathbf{w} ^{*}-\mathbf{w} ^{k+1}\Vert_{\mathbf{H}_{k}}^{2}-& \Vert\mathbf{w} ^{k}-\bar{\mathbf{w} }^{k}\Vert_{\mathbf{G}_{k}}^{2}\ge 2 \left \{\rho [ f(\bar{\mathbf{x}}^{k})- f(\mathbf{x}^{*})]+(\bar{\mathbf{w} }^{k}-\mathbf{w} ^{*})^{\top}\frac{1}{\eta_k}\boldsymbol{\Gamma}(\mathbf{w} ^{*})\right\}\ge0.\nonumber
\end{aligned}\end{equation}
Thus, this theorem holds.
\end{proof}

\begin{remark}
 The above theorem establishes a fundamental inequality for the sequence $\{\mathbf{w}^k, \bar{\mathbf{w}}^k, \mathbf{w}^{k+1}\}$ generated by the \textsf{Spice} method. The inequality shows that the norm of the difference between the iterate $\mathbf{w}^{k+1}$ and any optimal solution $\mathbf{w}^*$, measured in the $\mathbf{H}_k$-norm, decreases from one iteration to the next. Specifically, this decrease is driven by the residual term $\Vert \mathbf{w}^k - \bar{\mathbf{w}}^k \Vert_{\mathbf{G}_k}^2$. This term quantifies the gap between the predictor step $\mathbf{w}^k$ and the corrected point $\bar{\mathbf{w}}^k$, ensuring that the iterates progressively approach the set of optimal solutions $\Omega^*$ under the given convergence condition. Hence, the theorem guarantees the convergence of the \textsf{Spice} method as long as the matrix $\mathbf{M}_k$ satisfies condition (\ref{sec3-eq5}).
\end{remark}

\begin{lemma} \label{sec3-lemma3}
Let sequence $\{\mathbf{w}^k\}$ be generated by the \textsf{Spice} method. If parameters $\{r_{k},s_{k},\eta_{k}\}$ satify (\ref{sec3-eq4}) and (\ref{sec3-add5}), for any $k\ge1$, the sequence $\mathbf{H}_{k}$ is diagonal and monotonically non-increasing. We have $\mathbf{H}_{k} \preceq \mathbf{H}_{k-1}$ and 
\begin{equation}
\|\mathbf{w}^{k+1} - \mathbf{w}^*\|_{\mathbf{H}_{k+1}}^2 \leq  \|\mathbf{w}^{0} - \mathbf{w}^*\|_{\mathbf{H}_0}^2.\label{sec3-eq11}
\end{equation}
\end{lemma}

\begin{proof} 
According to the above condition, we have:
\begin{align}
\|\mathbf{w}^{k+1} - \mathbf{w}^*\|_{\mathbf{H}_{k+1}}^2 
&\leq \|\mathbf{w}^{k+1} - \mathbf{w}^*\|_{\mathbf{H}_k}^2 \nonumber\\
&\leq   \|\mathbf{w}^k - \mathbf{w}^*\|_{\mathbf{H}_k}^2 - \|\mathbf{w}^k - \bar{\mathbf{w}}^k\|_{\mathbf{G}_k}^2  \nonumber\\
&\leq \|\mathbf{w}^k - \mathbf{w}^*\|_{\mathbf{H}_k}^2 \nonumber\\
&\leq  \|\mathbf{w}^{0} - \mathbf{w}^*\|_{\mathbf{H}_0}^2 \nonumber
\end{align}
Thus, this lemma holds.
\end{proof}

\begin{theorem}\label{sec3-theorem2}
Let $\{\mathbf{w}^k,\bar{\mathbf{w}}^k\}$ the sequences generated by the \textsf{Spice} method. For the given optimal solution $\mathbf{w}^{*}$, the following limit equations hold
\begin{equation}
\lim_{k \to \infty} \|\mathbf{w}^k - \bar{\mathbf{w}}^k\|^2 = 0 \quad {\rm and}\quad \lim_{k \to \infty}\mathbf{w}^{k}=\mathbf{w}^{*}.\label{sec3-eq12}
\end{equation}
\end{theorem}
\begin{proof}
 We begin by analyzing the following inequality:
\begin{align}
\|\mathbf{w}^{k+1} - \mathbf{w}^*\|_{\mathbf{H}_{k+1}}^2 
&\leq \|\mathbf{w}^{k+1} - \mathbf{w}^*\|_{\mathbf{H}_k}^2\nonumber \\
&\leq  \|\mathbf{w}^k - \mathbf{w}^*\|_{\mathbf{H}_k}^2 - \|\mathbf{w}^k - \bar{\mathbf{w}}^k\|_{\mathbf{G}_k}^2\nonumber  \\
&\leq \|\mathbf{w}_0 - \mathbf{w}^*\|_{\mathbf{H}_0}^2  - \sum_{t=0}^{k} \|\mathbf{w}_t - \bar{\mathbf{w}}_t\|_{\mathbf{G}_t}^2.\nonumber 
\end{align}
This implies:
\begin{equation}
\sum_{t=0}^{k} \|\mathbf{w}_t - \bar{\mathbf{w}}_t\|_{\mathbf{G}_t}^2 \leq \|\mathbf{w}_0 - \mathbf{w}^*\|_{\mathbf{H}_0}^2-\|\mathbf{w}^{k+1} - \mathbf{w}^*\|_{\mathbf{H}_{k+1}}^2  .\nonumber 
\end{equation}
Summing over all iterations, we find:
\begin{equation}
\sum_{k=0}^{\infty} \|\mathbf{w}^k - \bar{\mathbf{w}}^k\|_{\mathbf{G}_k}^2 \leq \|\mathbf{w}_0 - \mathbf{w}^*\|_{\mathbf{H}_0}^2.\nonumber 
\end{equation}
Therefore, as $k \to \infty$:
\begin{equation}
\lim_{k \to \infty} \|\mathbf{w}^k - \bar{\mathbf{w}}^k\|_{\mathbf{G}_k}^2 = 0. \label{sec3-eq13}
\end{equation}
Since $\mathbf{G}_{k}$ depends on the values of $\{r_{k}, s_{k}\}$, equation (\ref{sec3-eq13}) further implies:
\begin{equation}
\lim_{k \to \infty} \|\mathbf{w}^k - \bar{\mathbf{w}}^k\|^2 = 0.\nonumber 
\end{equation}
Finally, as $k \to \infty$, inequality (\ref{sec3-eq2}) leads to:
\begin{equation}
\bar{\mathbf{w}}^{k}\in \Omega, \quad \rho[f(\mathbf{x}) - f(\bar{\mathbf{x}}^{k})] + (\mathbf{w} - \bar{\mathbf{w}}^{k})^{\top}\frac{1}{\eta_k}\boldsymbol{\Gamma}( \bar{\mathbf{w}}^{k})\ge0, \quad\forall \ \mathbf{w}\in \Omega,\nonumber 
\end{equation}
where $\bar{\mathbf{w}}^{k}$ satisfies the optimality condition. Consequently, we obtain:
\begin{equation}
\lim_{k \to \infty}\mathbf{w}^{k}=\bar{\mathbf{w}}^{k}=\mathbf{w}^{*}.\nonumber 
\end{equation}
Thus, this theorem is proved.
\end{proof}

\begin{remark}
This theorem strengthens the convergence properties of the \textsf{Spice} method. It asserts that the sequence $\{\mathbf{w}^k, \bar{\mathbf{w}}^k\}$ not only satisfies the diminishing residual condition $\|\mathbf{w}^k - \bar{\mathbf{w}}^k\|^2 \to 0$ as $k \to \infty$, but also ensures that the iterates $\mathbf{w}^k$ themselves converge to an optimal solution $\mathbf{w}^*$ in the limit. This implies that the discrepancy between the predictor and corrector steps vanishes asymptotically, guaranteeing that the algorithm progressively stabilizes and converges to an optimal solution. The theorem provides a rigorous foundation for the global convergence of the \textsf{Spice} method.
\end{remark}

\subsection{Ergodic Convergence Analysis}
To analyze the ergodic convergence, we require an alternative characterization of the solution set for the variational inequality (\ref{sec3-eq2}). This characterization is provided in the following lemma:
\begin{lemma}\label{sec3-lemma4}
The solution of the variational inequality (\ref{sec2-eq1}) can be characterized as
\begin{equation}
\Omega^{*}=\underset{\mathbf{w}\in \Omega}{\bigcap}\left\{\bar{\mathbf{w}}\in \Omega: \rho[\vartheta(\mathbf{u})-\vartheta(\bar{\mathbf{u}})]+(\mathbf{w}-\bar{\mathbf{w}})^{\top}\frac{1}{\eta}\boldsymbol{\Gamma}(\mathbf{w})\ge0\right\}.\label{sec3-eq14}
\end{equation}\label{lemma4}
\end{lemma}\vspace{-0.4cm}
\noindent The proof can be found in \cite{our1} (Lemma 3.3). By Lemma \ref{sec2-lemma2}, for any $ \mathbf{w} \in \Omega $, the monotone operator can be expressed as:
\begin{equation}
(\mathbf{w} - \mathbf{w}^{*})^{\top} \boldsymbol{\Gamma}(\mathbf{w}) \geq (\mathbf{w} - \mathbf{w}^{*})^{\top} \boldsymbol{\Gamma}(\mathbf{w}^{*}).
\label{sec3-eq15}
\end{equation}
Referring to \eqref{sec3-eq15}, the variational inequality (\ref{sec2-eq1}) can be rewritten as
\begin{equation}
\mathbf{w}^{*} \in \Omega, \quad \rho[\vartheta(\mathbf{u}) - \vartheta(\mathbf{u}^{*})] + (\mathbf{w} - \mathbf{w}^{*})^{\top} \frac{1}{\eta} \boldsymbol{\Gamma}(\mathbf{w}) \geq 0, \quad \forall \ \mathbf{w} \in \Omega.
\label{sec3-eq16}
\end{equation}
For a given $\epsilon > 0$, $\bar{\mathbf{w}} \in \Omega$ is considered an approximate solution to the optimal solution $\mathbf{w}^{*}$ if it satisfies the following inequality:
\begin{equation}
\rho[\vartheta(\mathbf{u}) - \vartheta(\bar{\mathbf{u}})] + (\mathbf{w} - \bar{\mathbf{w}})^{\top} \frac{1}{\eta} \boldsymbol{\Gamma}(\mathbf{w}) \geq -\epsilon, \quad \forall \ \mathbf{w} \in \mathsf{N}_{\epsilon}(\bar{\mathbf{w}}),
\end{equation}
where $\mathsf{N}_{\epsilon}(\bar{\mathbf{w}}) = \{\mathbf{w} \mid \|\mathbf{w} - \bar{\mathbf{w}}\| \leq \epsilon\}$ is the $\epsilon$-neighborhood of $\bar{\mathbf{w}}$. For a given $\epsilon > 0$, after $t$ iterations, this provides $\bar{\mathbf{w}} \in \Omega$ such that:
\begin{equation}
\sup_{\mathbf{w} \in \mathsf{N}_{\epsilon}(\bar{\mathbf{w}})} \left\{\rho[\vartheta(\bar{\mathbf{u}}) - \vartheta(\mathbf{u})] + (\bar{\mathbf{w}} - \mathbf{w})^{\top} \frac{1}{\eta} \boldsymbol{\Gamma}(\mathbf{w})\right\} \leq \epsilon.
\label{sec3-eq17}
\end{equation}
Referring to Lemma \ref{sec2-lemma2}, the inequality (\ref{sec3-eq6}) can be reformulated as
\begin{equation}
\rho \left[f(\mathbf{x})- f(\bar{\mathbf{x}}^{k})\right]+(\mathbf{w} -\bar{\mathbf{w} }^{k})^{\top}\frac{1}{\eta_k}\boldsymbol{\Gamma}(\mathbf{w} )+\frac{1}{2 }\|\mathbf{w} -\mathbf{w} ^{k}\|_{\mathbf{H}_{k}}^{2}\ge\frac{1}{2 }\|\mathbf{w} -\mathbf{w} ^{k+1}\|_{\mathbf{H}_{k}}^{2}, \quad\forall \ \mathbf{w} \in \Omega. \label{sec3-eq18}
\end{equation}
Note that the above inequality holds only for  $\mathbf{G}_{k}\succ 0$.

\begin{theorem}\label{sec3-theorem3}
Let $\{\bar{\mathbf{x}}^{k}\}$ be a sequence generated by the \textsf{Spice} method for problem \textsf{P}0 and $\eta_{k}$ satisfy condtion (\ref{sec3-add5}). New ariables $\bar{\mathbf{x}}_{t}$, $\eta_{t}$ and $\bar{\mathbf{w} }_{t}$ are defined as follows
\begin{equation}
\bar{\mathbf{x}}_{t}=\frac{1}{t+1}\sum_{k=0}^{t}\bar{\mathbf{x}}^{k},\quad \eta_{t}=\frac{t+1}{\sum_{k=0}^{t}\frac{1}{\eta_{k}}}, \quad\bar{\mathbf{w} }_{t}=\frac{\sum_{k=0}^t\frac{\bar{\mathbf{w} }^{k}}{\eta_k}}{\sum_{k=0}^t\frac{1}{\eta_k}}.\nonumber
\end{equation}
If Lemma \ref{sec3-lemma2} holds, for any iteration $t>0$, scaling factors $\rho(t)>0$ and $\eta_k >0$, the error bound satisfies
\begin{equation}
\mathbf{w}^{*}\in \Omega,\quad f(\bar{\mathbf{x}}_{t})-f(\mathbf{x}^{*})+(\bar{\mathbf{w} }_{t}-\mathbf{w}^{*} )^{\top}\frac{1}{\eta_t\rho(t)}\boldsymbol{\Gamma}(\mathbf{w}^{*} )\leq \frac{1}{2 \eta_0\rho(t)(t+1)}\Vert\mathbf{w}^{*} -\mathbf{w} ^{0}\Vert_{\mathsf{H}_{0}}^{2}, \quad\forall \ \bar{\mathbf{w} }_{t} \in \Omega. \label{sec3-eq19}
\end{equation}
\end{theorem}
\begin{proof}
Since the sequence $\{\bar{\mathbf{w} }^{k},k=0, 1,\cdots,t\}$ belongs to the convex set $\Omega$, its linear combination $\bar{\mathbf{w} }_{t}$ also belongs to $ \Omega$. Summing inequality \eqref{sec3-eq18} over all iterations, for any $ \mathbf{w} \in \Omega$, we obtain: 
\begin{equation}\begin{aligned}
\rho(t)(t+1)f(\mathbf{x})- \rho(t)\sum_{k=0}^{t}f(\bar{\mathbf{x}}^{k})+\sum_{k=0}^{t}\frac{1}{\eta_k}\left(\mathbf{w} -\bar{\mathbf{w} }^{k}\right)^{\top}\boldsymbol{\Gamma}(\mathbf{w} )&+\frac{1}{2 }\sum_{k=0}^{t}\|\mathbf{w} -\mathbf{w} ^{k}\|_{\mathbf{H}_{k}}^{2}\ge\frac{1}{2 }\sum_{k=0}^{t}\|\mathbf{w} -\mathbf{w} ^{k+1}\|_{\mathbf{H}_{k}}^{2}.\nonumber
\end{aligned}\end{equation}
The inequality above divided by $t+1$ equals
\begin{equation}\begin{aligned}
\rho(t) f(\mathbf{x})- \frac{\rho(t) }{t+1}\sum_{k=0}^{t}f(\bar{\mathbf{x}}^{k})+\left(\mathbf{w} -\bar{\mathbf{w} }_{t}\right)^{\top}\frac{1}{\eta_t}\boldsymbol{\Gamma}(\mathbf{w} )+&\frac{1}{2 (t+1)}\sum_{k=0}^{t}\|\mathbf{w} -\mathbf{w} ^{k}\|_{\mathbf{H}_{k}}^{2}\ge\frac{1}{2 (t+1)}\sum_{k=0}^{t}\|\mathbf{w} -\mathbf{w} ^{k+1}\|_{\mathbf{H}_{k}}^{2}.
 \end{aligned}\nonumber
 \end{equation}
Based on the fact that $f(\mathbf{x})$ is convex, the inequality  $f(\bar{\mathbf{x}}_{t})\leq\frac{1}{t+1}\sum_{k=0}^{t}f(\bar{\mathbf{x}}^{k})$ holds.
\begin{equation}\begin{aligned}
\rho(t) [f(\mathbf{x})- f(\bar{\mathbf{x}}_{t})] +\left(\mathbf{w} -\bar{\mathbf{w} }_{t}\right)^{\top}\frac{1}{\eta_t}\boldsymbol{\Gamma}(\mathbf{w} )+\frac{1}{2 (t+1)}\sum_{k=0}^{t}\big[\|\mathbf{w} -\mathbf{w} ^{k}\|_{\mathbf{H}_{k}}^{2}-\|\mathbf{w} -\mathbf{w} ^{k+1}\|_{\mathbf{H}_{k}}^{2}\big]&\ge 0.
 \end{aligned}
\nonumber
\end{equation}
Setting $\mathbf{w}=\mathbf{w}^{*}$, the above inequality can be expressed as
\begin{align}
\rho(t) [f(\bar{\mathbf{x}}_{t})-f(\mathbf{x}^{*})] +\left(\bar{\mathbf{w} }_{t}-\mathbf{w}^{*}\right)^{\top}\frac{1}{\eta_t}\boldsymbol{\Gamma}(\mathbf{w}^{*} )&\le \frac{1}{2 (t+1)}\sum_{k=0}^{t}\left[\|\mathbf{w}^{*} -\mathbf{w} ^{k}\|_{\mathbf{H}_{k}}^{2}-\|\mathbf{w}^{*} -\mathbf{w} ^{k+1}\|_{\mathbf{H}_{k}}^{2}\right],\nonumber\\
\rho(t) [f(\bar{\mathbf{x}}_{t})-f(\mathbf{x}^{*})]+\left(\bar{\mathbf{w} }_{t}-\mathbf{w}^{*}\right)^{\top}\frac{1}{\eta_t}\boldsymbol{\Gamma}(\mathbf{w}^{*} )&\le\frac{1}{2(t+1) }\left[\|\mathbf{w}^{*} -\mathbf{w} ^{0}\|_{\mathbf{H}_{0}}^{2}-\sum_{k=0}^{t-1}\|\mathbf{w}^{*} -\mathbf{w} ^{k}\|_{\mathbf{D}_{k}}^{2} \right].
\label{sec3-eq20}
 \end{align}
Since $\mathbf{D}_{k}\succeq0, (k=0, \cdots, t-1)$ and $\mathbf{H}_{0} = \frac{1}{\eta_0}\mathsf{H}_{0}$, the inequality (\ref{sec3-eq20}) can be further rewritten as 
\begin{align}
f(\bar{\mathbf{x}}_{t})-f(\mathbf{x}^{*})+\left(\bar{\mathbf{w} }_{t}-\mathbf{w}^{*}\right)^{\top}\frac{1}{\eta_t\rho(t)}\boldsymbol{\Gamma}(\mathbf{w}^{*} )&\le\frac{1}{2 \eta_0\rho(t)(t+1)}\|\mathbf{w}^{*} -\mathbf{w} ^{0}\|_{\mathsf{H}_{0}}^{2}.\nonumber
 \end{align}
Thus, this theorem holds.
\end{proof}

\begin{remark}
The above theorem shows that the \textsf{Spice} method achieves an ergodic convergence rate of $\mathcal{O}(1/[\rho(t)(t+1)])$. The inequality (\ref{sec3-eq20}) provides an upper bound on the error, which is influenced by the scaling factor $\rho(t)$, the iteration number $t$, and the initial conditions of the variables. This insight highlights how different choices of $\rho(t)$ directly impact the convergence behavior of the method.
\end{remark}

\begin{corollary}
For any iteration $t > 0$, the choices of the scaling factor $\rho(t)$ allows for flexible control over the convergence rate:
\begin{enumerate}
\item For any $\alpha > 0$, setting $\rho(t) = (t+1)^{\alpha}$ results in an ergodic power-law convergence rate of $\mathcal{O}(1/(t+1)^{1+\alpha})$. This demonstrates that increasing $\alpha$ enhances the convergence rate, leading to faster decay of the error.
\item For any $\beta > 0$, setting $\rho(t) = e^{\beta t}$ results in an ergodic exponential convergence rate of $\mathcal{O}(1/[e^{\beta t}(t+1)])$. This choice accelerates convergence even further as the error decays exponentially with increasing iterations.
\item Setting $\rho(t) = (t+1)^{t}$ yields an ergodic power-exponential convergence rate of $\mathcal{O}(1/(t+1)^{t+1})$. This setting combines both power-law and exponential decay, offering a robust convergence rate that rapidly diminishes the error as $t$ increases.
\end{enumerate}
As $\rho(t) \rightarrow \infty$, the error term $f(\bar{\mathbf{x}}_{t}) - f(\mathbf{x}^{*}) \rightarrow 0_{+}$, ensuring the \textsf{Spice} method's asymptotic convergence to the optimal solution.
\end{corollary}

\subsection{Solving Convex Problems with Equality Constraints}
For the single variable, the convex problems with equality constraints can be written as:
\begin{equation}
\begin{aligned}
 \min\left\{ f(\mathbf{x}) \mid   \phi_{i}(\mathbf{x}) \leq 0, \mathbf{A}\mathbf{x}=\mathbf{b}, \ \mathbf{x}\in \mathcal{X}, \ i=1,\cdots,p_{1}\right\},
\end{aligned}\nonumber
\end{equation}
where $\mathbf{A} \in \mathbb{R}^{p_{2} \times n}$, $\mathbf{b} \in \mathbb{R}^{p_{2}}$ ($p_{1} + p_{2} = p$), and $\Phi(\mathbf{x}) = [\phi_{1}(\mathbf{x}), \cdots, \phi_{p_{1}}(\mathbf{x}), \phi_{p_{1} + 1}(\mathbf{x}), \cdots, \phi_{p}(\mathbf{x})]^{\top}$, with $[\phi_{p_{1} + 1}(\mathbf{x}), \cdots, \phi_{p}(\mathbf{x})]^{\top} = \mathbf{A} \mathbf{x} - \mathbf{b}$. The domain of the dual variable $\boldsymbol{\lambda}$ becomes $\mathcal{Z} := \mathbb{R}^{p_{1}}_{+} \times \mathbb{R}^{p_{2}}$.

\section{Nonlinear Convex Problems with Separable Variables} 
In this section, the \textsf{Spice} method is extended to address the following separable convex optimization problem, which includes nonlinear inequality constraints:
\begin{equation}
\begin{aligned}
 \mathsf{P}2: \quad\min\left\{ f(\mathbf{x})+g(\mathbf{y}) \ \middle|\   \phi_{i}(\mathbf{x}) +\psi_{i}(\mathbf{y})\leq 0,\ \mathbf{x}\in \mathcal{X},\ \mathbf{y}\in \mathcal{Y}, \ i=1,\cdots,p\right\},
\end{aligned}\label{sec4-eq1}
\end{equation}
where two conve set $\mathcal{X}\in\mathbb{R}^{n}$ and $\mathcal{Y}\in\mathbb{R}^{m}$ are nonempty and closed. Two convex objective functions $\{f(\mathbf{x}): \mathbb{R}^{n}\rightarrow \mathbb{R}\}$ and $\{g(\mathbf{y}): \mathbb{R}^{m}\rightarrow \mathbb{R}\}$ are proper and closed. The constraint functions $\{\phi_{i}: \mathbb{R}^{n}\rightarrow \mathbb{R},\; i=1,\ldots,p\}$ and $\{\psi_{i}: \mathbb{R}^{m}\rightarrow \mathbb{R},\; i=1,\ldots,p\}$ are convex and continuously differentiable. The Lagrangian function associated with \textsf{P}2 is characterized as:
\begin{equation}
\mathcal{L}(\mathbf{x},\mathbf{y},\boldsymbol{\lambda})= f(\mathbf{x})+g(\mathbf{y})+\boldsymbol{\lambda}^{\top} [\Phi(\mathbf{x}) + \Psi(\mathbf{y})], \label{sec4-eq2}
\end{equation} 
where  $\Phi(\mathbf{x})=[\phi_{1}(\mathbf{x}),\ldots,\phi_{p}(\mathbf{x})]^{\top}$  and $\Psi(\mathbf{y})=[\psi_{1}(\mathbf{y}),\cdots,\psi_{p_{1}}(\mathbf{y})]^{\top}$. The dual variable $\boldsymbol{\lambda}$ belongs to the set $\mathcal{Z}:=\mathbb{R}^{p}_{+}$. In fact, constrained convex optimization can be interpreted as a specific instance of a saddle point problem:
\begin{equation}
 \min_{\mathbf{x},\mathbf{y}} \max_{\boldsymbol{\lambda}} \left\{\mathcal{L}(\mathbf{x},\mathbf{y},\boldsymbol{\lambda})\mid \mathbf{x}\in \mathcal{X},\mathbf{y}\in \mathcal{Y}, \boldsymbol{\lambda}\in \mathcal{Z}\right\}.\label{sec4-eq3}
\end{equation}

\noindent By Lemma \ref{sec2-lemma1}, for the saddle point, the following variational inequality holds:
\begin{empheq}[left={}\empheqlbrace]{alignat=2}
     \mathbf{x}^{*}\in \mathcal{X},\quad && f(\mathbf{x})- f(\mathbf{x}^{*})+(\mathbf{x}-\mathbf{x}^{*})^{\top}\mathcal{D} \Phi (\mathbf{x}^{*})^{\top}\boldsymbol{\lambda}^{*}\ge0,&\quad\forall \ \mathbf{x}\in \mathcal{X},\nonumber\\[0.1cm]
       \mathbf{y}^{*}\in \mathcal{Y},\quad && g(\mathbf{y})- g(\mathbf{y}^{*})+(\mathbf{y}-\mathbf{y}^{*})^{\top}\mathcal{D} \Psi (\mathbf{y}^{*})^{\top}\boldsymbol{\lambda}^{*}\ge0,&\quad\forall \ \mathbf{y}\in \mathcal{Y},\nonumber\\[0.1cm]
     \boldsymbol{\lambda}^{*}\in  \mathcal{Z}, \quad&&(\boldsymbol{\lambda}-\boldsymbol{\lambda}^{*})^{\top} [- \Phi (\mathbf{x}^{*})-\Psi (\mathbf{y}^{*})]\ge0,&\quad \forall \ \boldsymbol{\lambda}\in  \mathcal{Z},\nonumber
        \end{empheq}
where $\mathcal{D} \Phi (\mathbf{x})=[\nabla \phi_{1},\cdots,\nabla \phi_{p}]^{\top}\in \mathbb{R}^{p\times n}$, $\mathcal{D} \Psi (\mathbf{y})=[\nabla \psi_{1},\cdots,\nabla \psi_{p}]^{\top}\in \mathbb{R}^{p\times m}$. 
The above inequalities can be further characterized as a monotone variational inequality:
\begin{equation}
 \mathbf{w}^{*}\in \Omega, \quad \vartheta(\mathbf{u})- \vartheta(\mathbf{u}^{*})+(\mathbf{w}-\mathbf{w}^{*})^{\top}\boldsymbol{\Gamma}(\mathbf{w}^{*})\ge0, \quad\forall \ \mathbf{w}\in \Omega, \label{sec4-eq4}
\end{equation}
where $\vartheta(\mathbf{u})=f(\mathbf{x})+g(\mathbf{y})$, 
\begin{equation}\label{sec4-eq5}
\mathbf{u}=\left( \begin{array}{c} \mathbf{x} \\[0.2cm]
 \mathbf{y} \end{array} \right), \quad
 \mathbf{w}=\left( \begin{array}{c} \mathbf{x} \\[0.2cm]
 \mathbf{y} \\[0.2cm]
      \boldsymbol{\lambda} \end{array} \right), \quad
 \boldsymbol{\Gamma}(\mathbf{w})=\left( \begin{array}{c} \mathcal{D} \Phi (\mathbf{x})^{\top}\boldsymbol{\lambda} \\[0.2cm]
 \mathcal{D} \Psi (\mathbf{y})^{\top}\boldsymbol{\lambda} \\[0.2cm]
        -\Phi (\mathbf{x}) -\Psi (\mathbf{y}) \end{array} \right),\quad \mbox{ and } \quad\Omega=\mathcal{X}\times \mathcal{Y}\times\mathcal{Z}. 
\end{equation}
 Then, the following lemma shows that the operator $\boldsymbol{\Gamma}$ is monotone.

\begin{lemma}\label{sec4-lemma1}
Let $\mathcal{X}\subset \mathbb{R}^{n}$, $\mathcal{Y}\subset \mathbb{R}^{m}$, $ \mathcal{Z}:=\mathbb{R}^{p}_{+}$ be closed convex sets. Then the operator $\boldsymbol{\Gamma}$ defined in (\ref{sec4-eq5}) satisfies 
\begin{equation}
(\mathbf{w}-\bar{\mathbf{w}})^{\top}[\boldsymbol{\Gamma}(\mathbf{w})-\boldsymbol{\Gamma}(\bar{\mathbf{w}})]\ge0, \quad\forall \ \mathbf{w}, \bar{\mathbf{w}}\in \mathcal{X}\times\mathcal{Y}\times \mathcal{Z}.\label{sec4-eq6}
\end{equation}
\end{lemma}
\begin{proof}
Recall that $\{\phi_{i}$ and $\psi_{i}\;(i=1,\ldots,p)\}$ are assumed to be convex and differentiable on $\mathcal{X}$ and $\mathcal{Y}$. Then for any $\mathbf{x},\;\bar{\mathbf{x}}\in\mathcal{X}$ and $\mathbf{y},\;\bar{\mathbf{y}}\in\mathcal{Y}$, we have 
\begin{equation*}
  \phi_{i}(\mathbf{x})\ge\phi_{i}(\bar{\mathbf{x}})+\langle\nabla\phi_{i}(\bar{\mathbf{x}}),\mathbf{x}-\bar{\mathbf{x}}\rangle, \quad i=1,\ldots,p,
\end{equation*}
and 
\begin{equation*}
\psi_{i}(\mathbf{y})\ge\psi_{i}(\bar{\mathbf{y}})+\langle\nabla\psi_{i}(\bar{\mathbf{y}}),\mathbf{y}-\bar{\mathbf{y}}\rangle, \quad i=1,\ldots,p,
\end{equation*}
It implies that 
$$\theta_i'(\mathbf{x}\mid \bar{\mathbf{x}}):=\phi_{i}(\mathbf{x})-\phi_{i}(\bar{\mathbf{x}})+\langle\nabla\phi_{i}(\bar{\mathbf{x}}),\mathbf{x}-\bar{\mathbf{x}}\rangle\ge0, \quad i=1,\ldots,p,$$
and
$$\theta_i''(\mathbf{y}\mid \bar{\mathbf{y}}):=\psi_{i}(\mathbf{y})-\psi_{i}(\bar{\mathbf{y}})+\langle\nabla\phi_{i}(\bar{\mathbf{y}}),\mathbf{y}-\bar{\mathbf{y}}\rangle\ge0, \quad i=1,\ldots,p.$$ 
In the same way, for any $i\in\{1,\ldots,p\}$, we obtain $\theta_i'(\bar{\mathbf{x}}\mid \mathbf{x})\ge0$ and $\theta_i''(\bar{\mathbf{y}}\mid \mathbf{y})\ge0$. Let $\boldsymbol{\Theta}':=[\theta_{1}',\cdots,\theta_{p}']^{\top}$ and $\boldsymbol{\Theta}'':=[\theta_{1}'',\cdots,\theta_{p}'']^{\top}$. Since $\boldsymbol{\lambda},\bar{\boldsymbol{\lambda}}\in \mathcal{Z}$, $\boldsymbol{\Theta}'(\mathbf{x}\mid\bar{\mathbf{x}}),\boldsymbol{\Theta}'(\bar{\mathbf{x}}\mid\mathbf{x}),\boldsymbol{\Theta}''(\mathbf{y}\mid\bar{\mathbf{y}}),\boldsymbol{\Theta}''(\bar{\mathbf{y}}\mid\mathbf{y})\in \mathbb{R}^{p}_{+}$, the following inequality holds.
\begin{equation*}
\begin{aligned}
(\mathbf{w}-\bar{\mathbf{w}})^{\top}[\boldsymbol{\Gamma}(\mathbf{w})-\boldsymbol{\Gamma}(\bar{\mathbf{w}})]
=&\left( \begin{array}{c} \mathbf{x} -\bar{\mathbf{x}}\\ [0.2cm]
\mathbf{y} -\bar{\mathbf{y}}\\ [0.2cm]
\boldsymbol{\lambda}-\bar{\boldsymbol{\lambda}} \end{array} \right)^{\top}\left( \begin{array}{c} \mathcal{D} \Phi (\mathbf{x})^{\top}\boldsymbol{\lambda} -\mathcal{D} \Phi (\bar{\mathbf{y}})^{\top}\bar{\boldsymbol{\lambda}}\\[0.2cm]
\mathcal{D} \Psi (\mathbf{y})^{\top}\boldsymbol{\lambda} -\mathcal{D} \Psi (\bar{\mathbf{x}})^{\top}\bar{\boldsymbol{\lambda}}\\[0.2cm]
-  \Phi (\mathbf{x}) +  \Phi (\bar{\mathbf{x}})-  \Psi (\mathbf{y}) +  \Psi (\bar{\mathbf{y}})\end{array} \right)\\[0.1cm]
 =&\boldsymbol{\lambda}^{\top}[  \Phi (\bar{\mathbf{x}})-  \Phi (\mathbf{x})-\mathcal{D} \Phi (\mathbf{x})(\bar{\mathbf{x}}-\mathbf{x} ) ]+\bar{\boldsymbol{\lambda}}^{\top} [  \Phi (\mathbf{x}) -  \Phi (\bar{\mathbf{x}})-\mathcal{D} \Phi (\bar{\mathbf{x}})(\mathbf{x} -\bar{\mathbf{x}})]\\[0.1cm]
 &+\boldsymbol{\lambda}^{\top}[  \Psi (\bar{\mathbf{y}})-  \Psi (\mathbf{y})-\mathcal{D} \Psi (\mathbf{y})(\bar{\mathbf{y}}-\mathbf{y} ) ]+\bar{\boldsymbol{\lambda}}^{\top} [  \Psi (\mathbf{y}) -  \Psi (\bar{\mathbf{y}})-\mathcal{D} \Psi (\bar{\mathbf{y}})(\mathbf{y} -\bar{\mathbf{y}})]\\[0.1cm]
 =&\boldsymbol{\lambda}^{\top}\boldsymbol{\Theta}'(\bar{\mathbf{x}}\mid\mathbf{x})+\bar{\boldsymbol{\lambda}}^{\top}\boldsymbol{\Theta}'(\mathbf{x}\mid\bar{\mathbf{x}})+\boldsymbol{\lambda}^{\top}\boldsymbol{\Theta}''(\bar{\mathbf{y}}\mid\mathbf{y})+\bar{\boldsymbol{\lambda}}^{\top}\boldsymbol{\Theta}''(\mathbf{y}\mid\bar{\mathbf{y}})\ge 0.
\end{aligned}
\end{equation*}
Thus, this lemma holds.
\end{proof}

\subsection{PPA-Like Prediction Scheme}\label{section3}
For the prediction scheme, we take the output of PPA $ \bar{\mathbf{w}}^{k}$ as the predictive variables. Consider the scaling factors $\rho > 0$ and $\eta_k > 0$. The corresponding scaling Lagrangian function for \textsf{P}2 is defined as 
\begin{equation}
\mathcal{L}(\mathbf{x},\mathbf{y},\boldsymbol{\lambda},\rho,\eta)= \rho[f(\mathbf{x})+g(\mathbf{y})]+\boldsymbol{\lambda}^{\top}\frac{1}{\eta} [\Phi(\mathbf{x}) + \Psi(\mathbf{y})],\label{sec4-eq7}
\end{equation} 
For the $k$-th iteration, the PPA-like prediction scheme utilizes the current estimates $(\mathbf{x}^{k},\mathbf{y}^{k},\boldsymbol{\lambda}^{k})$ to sequentially update the predicted values  $(\bar{\mathbf{x}}^{k}$, $\bar{\mathbf{y}}^{k}$, and $\bar{\boldsymbol{\lambda}}^{k})$ by solving the following subproblems:
\begin{subequations}\label{sec4-eq8}
\begin{empheq}[left={}\empheqlbrace]{alignat=2}
\bar{\mathbf{x}}^{k}=&\arg\min\{\mathcal{L}(\mathbf{x},\mathbf{y}^{k},\boldsymbol{\lambda}^{k},\rho,\eta_k)+\frac{r_{k}}{2}\|\mathbf{x}-\mathbf{x}^{k}\|^{2}\mid \mathbf{x}\in \mathcal{X}\},\\
\bar{\mathbf{y}}^{k}=&\arg\min\{\mathcal{L}(\bar{\mathbf{x}}^{k},\mathbf{y},\boldsymbol{\lambda}^{k},\rho,\eta_k)+\frac{r_{k}}{2}\|\mathbf{y}-\mathbf{y}^{k}\|^{2}\mid \mathbf{y}\in \mathcal{Y}\},\\
\bar{\boldsymbol{\lambda}}^{k}=&\arg\max\{\mathcal{L}(\bar{\mathbf{x}}^{k},\bar{\mathbf{y}}^{k},\boldsymbol{\lambda},\rho,\eta_k)-\frac{s_{k}}{2}\Vert\boldsymbol{\lambda}-\boldsymbol{\lambda}^{k}\Vert^{2}\mid \boldsymbol{\lambda}\in  \mathcal{Z}\}.
\end{empheq}\end{subequations}
As established in Lemma \ref{sec2-lemma1}, these optimization problems (\ref{sec4-eq8}) can be expressed as the following variational inequalities:
\begin{empheq}[left={}\empheqlbrace]{alignat=2}
     \bar{\mathbf{x}}^{k}\in \mathcal{X},\quad&&  \rho \left[f(\mathbf{x})- f(\bar{\mathbf{x}}^{k})\right]+(\mathbf{x}-\bar{\mathbf{x}}^{k})^{\top}\left[\frac{1}{\eta_k}\mathcal{D} \Phi (\bar{\mathbf{x}}^{k})^{\top}\boldsymbol{\lambda}^{k}+r_{k}(\bar{\mathbf{x}}^{k}-\mathbf{x}^{k})\right]\ge0,&\quad \forall \ \mathbf{x}\in \mathcal{X},\nonumber\\[0.1cm]
          \bar{\mathbf{y}}^{k}\in \mathcal{Y},\quad&&  \rho [g(\mathbf{y})- g(\bar{\mathbf{y}}^{k})]+(\mathbf{y}-\bar{\mathbf{y}}^{k})^{\top}\left[\frac{1}{\eta_k}\mathcal{D} \Psi (\bar{\mathbf{y}}^{k})^{\top}\boldsymbol{\lambda}^{k}+r_{k}(\bar{\mathbf{y}}^{k}-\mathbf{y}^{k})\right]\ge0,&\quad \forall \ \mathbf{y}\in \mathcal{Y},\nonumber\\[0.1cm]
     \bar{\boldsymbol{\lambda}}^{k}\in  \mathcal{Z},\quad&& (\boldsymbol{\lambda}-\bar{\boldsymbol{\lambda}}^{k})^{\top}\left[ -\frac{1}{\eta_k}[\Phi (\bar{\mathbf{x}}^{k})+\Psi (\bar{\mathbf{y}}^{k})]+s_{k}(\bar{\boldsymbol{\lambda}}^{k}-\boldsymbol{\lambda}^{k})\right]\ge0,&\quad \forall \ \boldsymbol{\lambda}\in  \mathcal{Z}.\nonumber
\end{empheq}
In light of the optimality conditions, we can further refine these inequalities into the form:
\begin{empheq}[left={}\empheqlbrace]{alignat=2}
     \bar{\mathbf{x}}^{k}\!\in\! \mathcal{X},\ &&  \rho\! \left[f(\mathbf{x})\!-\! f(\bar{\mathbf{x}}^{k})\right]\!+\!(\mathbf{x}\!-\!\bar{\mathbf{x}}^{k})^{\top}\left[\frac{1}{\eta_k}\mathcal{D} \Phi (\bar{\mathbf{x}}^{k})^{\top}\bar{\boldsymbol{\lambda}}^{k}\!+\!r_{k}(\bar{\mathbf{x}}^{k}\!-\!\mathbf{x}^{k})\!-\!\frac{1}{\eta_k}\mathcal{D} \Phi (\bar{\mathbf{x}}^{k})^{\top}(\bar{\boldsymbol{\lambda}}^{k}\!-\!\boldsymbol{\lambda}^{k})\right]\ge0,&\  \forall \ \mathbf{x}\in \mathcal{X},\nonumber\\[0.1cm]
          \bar{\mathbf{y}}^{k}\!\in\! \mathcal{Y},\ &&  \rho [g(\mathbf{y})\!-\! g(\bar{\mathbf{y}}^{k})]\!+\!(\mathbf{y}\!-\!\bar{\mathbf{y}}^{k})^{\top}\left[\frac{1}{\eta_k}\mathcal{D} \Psi (\bar{\mathbf{y}}^{k})^{\top}\bar{\boldsymbol{\lambda}}^{k}\!+\!r_{k}(\bar{\mathbf{y}}^{k}\!-\!\mathbf{y}^{k})\!-\!\frac{1}{\eta_k}\mathcal{D} \Psi (\bar{\mathbf{y}}^{k})^{\top}(\bar{\boldsymbol{\lambda}}^{k}\!-\!\boldsymbol{\lambda}^{k})\right]   \ge0,&\  \forall \ \mathbf{y}\in \mathcal{Y},\nonumber\\[0.1cm]
     \bar{\boldsymbol{\lambda}}^{k}\!\in\!  \mathcal{Z},\ && (\boldsymbol{\lambda}\!-\!\bar{\boldsymbol{\lambda}}^{k})^{\top}\left[ \!-\!\frac{1}{\eta_k}[\Phi (\bar{\mathbf{x}}^{k})\!+\!\Psi (\bar{\mathbf{y}}^{k})]\!+\!s_{k}(\bar{\boldsymbol{\lambda}}^{k}\!-\!\boldsymbol{\lambda}^{k})\right]\ge0,&\  \forall \ \boldsymbol{\lambda}\in  \mathcal{Z}.\nonumber
\end{empheq}
The above inequalities above can be combined into
\begin{equation}
   \rho [\vartheta(\mathbf{u})-\vartheta(\bar{\mathbf{u}}^{k})]+(\mathbf{w}- \bar{\mathbf{w}}^{k})^{\top}\left\{\frac{1}{\eta_k}\boldsymbol{\Gamma}( \bar{\mathbf{w}}^{k})+\left(\begin{array}{r}
r_{k}(\bar{\mathbf{x}}^{k}-\mathbf{x}^{k}) -\frac{1}{\eta_k}\mathcal{D} \Phi (\bar{\mathbf{x}}^{k})^{\top}(\bar{\boldsymbol{\lambda}}^{k}-\boldsymbol{\lambda}^{k}) \\[0.2cm]
r_{k}(\bar{\mathbf{y}}^{k}-\mathbf{y}^{k}) -\frac{1}{\eta_k}\mathcal{D} \Psi (\bar{\mathbf{y}}^{k})^{\top}(\bar{\boldsymbol{\lambda}}^{k}-\boldsymbol{\lambda}^{k}) \\[0.2cm]
s_{k}(\bar{\boldsymbol{\lambda}}^{k}-\boldsymbol{\lambda}^{k})
\end{array}
\right)\right\}\ge0, \quad\forall \ \mathbf{w}\in \Omega.\label{eq27 }\nonumber
\end{equation}
Further, it can be written as
\begin{equation}
 \rho [\vartheta(\mathbf{u})-\vartheta(\bar{\mathbf{u}}^{k})]+(\mathbf{w}- \bar{\mathbf{w}}^{k})^{\top}\left[\frac{1}{\eta_k}\boldsymbol{\Gamma}( \bar{\mathbf{w}}^{k})+\mathbf{Q}_{k}(\bar{\mathbf{w}}^{k}-\mathbf{w}^{k})\right]\ge0, \quad\forall \ \mathbf{w}\in \Omega,\label{sec4-eq9}
\end{equation}
where the proximal matrix is 
\begin{equation} \mathbf{Q}_{k}=\left(
\begin{array}{ccc}
r_{k}\mathbf{I}_{n} & \boldsymbol{0}&-\frac{1}{\eta_k}\mathcal{D} \Phi (\bar{\mathbf{x}}^{k})^{\top} \\[0.2cm]
\boldsymbol{0}&r_{k}\mathbf{I}_{m} &-\frac{1}{\eta_k}\mathcal{D} \Psi (\bar{\mathbf{y}}^{k})^{\top} \\[0.2cm]
\boldsymbol{0}&\boldsymbol{0}&s_{k}\mathbf{I}_{p}
\end{array}
\right),\nonumber
\end{equation}
and it can be considered a second-order diagonal scalar upper triangular block matrix. The variational inequality (\ref{sec4-eq9}) equals
\begin{equation}
 \bar{\mathbf{w}}^{k}\in \Omega, \quad \rho [\vartheta(\mathbf{u})- \vartheta(\bar{\mathbf{u}}^{k})]+(\mathbf{w}-\bar{\mathbf{w}}^{k})^{\top}\frac{1}{\eta_k}\boldsymbol{\Gamma}(\bar{\mathbf{w}}^{k})\ge(\mathbf{w}-\bar{\mathbf{w}}^{k})^{\top}\mathbf{Q}_{k}(\mathbf{w}^{k}-\bar{\mathbf{w}}^{k}), \quad\forall \ \mathbf{w}\in \Omega, \label{sec4-eq10}
\end{equation}
where the proximal matrix $\mathbf{Q}_{k}$ in the \textsf{Spice} method is called the predictive matrix.

\subsection{Matrix-Driven Correction Scheme}\label{section4}

For the correction scheme, we use a corrective matrix $\mathbf{M}_{k}$ to correct the predictive variables by the following equation.
\begin{equation}
\mathbf{w}^{k+1}=\mathbf{w}^{k}-\mathbf{M}_{k}(\mathbf{w}^{k}-\bar{\mathbf{w}}^{k}),\label{sec4-eq11}
\end{equation}
where the corrective matrix $\mathbf{M}_{k}$ is not unique. In this paper, we take
\begin{equation} \mathbf{M}_{k}=\left(
\begin{array}{ccc}
\mathbf{I}_{n} &\mathbf{0}& -\frac{1}{\eta_k r_{k}}\mathcal{D} \Phi (\bar{\mathbf{x}}^{k})^{\top} \\[0.2cm]\mathbf{0}&\mathbf{I}_{m}& -\frac{1}{\eta_k r_{k}}\mathcal{D} \Psi (\bar{\mathbf{y}}^{k})^{\top}\\[0.2cm]
\mathbf{0}&\mathbf{0}&\mathbf{I}_{p}
\end{array}\right), \quad\mathbf{M}_{k}^{-1}=\left(
\begin{array}{ccc}
\mathbf{I}_{n} &\mathbf{0}& \frac{1}{\eta_k r_{k}}\mathcal{D} \Phi (\bar{\mathbf{x}}^{k})^{\top} \\[0.2cm]\mathbf{0}&\mathbf{I}_{m}& \frac{1}{\eta_k r_{k}}\mathcal{D} \Psi (\bar{\mathbf{y}}^{k})^{\top}\\[0.2cm]
\mathbf{0}&\mathbf{0}&\mathbf{I}_{p}
\end{array}\right).\nonumber
\end{equation}
For any $r_{k}, s_{k}>0$, the extended matrix $\mathbf{H}_{k}$ is positive-definite:
\begin{equation} \mathbf{H}_{k}=\mathbf{Q}_{k}\mathbf{M}_{k}^{-1}=\left(
\begin{array}{ccc}
r_{k}\mathbf{I}_{n} & \mathbf{0} &\mathbf{0}\\[0.2cm]
\mathbf{0}&r_{k}\mathbf{I}_{m}&\mathbf{0}\\[0.2cm]
\mathbf{0}&\mathbf{0}&s_{k}\mathbf{I}_{p}
\end{array}
\right)\succ0.\nonumber
\end{equation}
When the condition $r_{k}s_{k}>\frac{1}{\eta_k^{2}}\|\mathcal{D} \Phi (\bar{\mathbf{x}}^{k})\|^{2}+\frac{1}{\eta_k^{2}}\|\mathcal{D} \Psi (\bar{\mathbf{y}}^{k})\|^{2}$ holds, we have
\begin{equation}\begin{aligned}
\mathbf{G}_{k}&=\mathbf{Q}_{k}^{\top}+\mathbf{Q}_{k}-\mathbf{M}_{k}^{\top}\mathbf{H}_{k}\mathbf{M}_{k}=\left(
\begin{array}{ccc}
r_{k}\mathbf{I}_{n} & \mathbf{0}&\mathbf{0} \\[0.2cm] \mathbf{0}&r_{k}\mathbf{I}_{m} &\mathbf{0}\\[0.2cm] \mathbf{0}&\mathbf{0}&s_{k}\mathbf{I}_{p} -\frac{1}{\eta_k^{2} r_{k}}[\mathcal{D} \Phi (\bar{\mathbf{x}}^{k})\mathcal{D} \Phi (\bar{\mathbf{x}}^{k})^{\top}+\mathcal{D} \Psi (\bar{\mathbf{y}}^{k})\mathcal{D} \Psi (\bar{\mathbf{y}}^{k})^{\top}]
\end{array}
\right)\succ0.\nonumber\end{aligned}
\end{equation}
For $\mu>1$ and $k\ge1$, we set $\mathsf{R}(\mathbf{u}^{k})=\|\mathcal{D} \Phi (\mathbf{x}^{k})\|^{2}+\|\mathcal{D} \Psi (\mathbf{y}^{k})\|^{2}$ and take a non-decreasing sequence like
\begin{subequations}\label{sec4-eq12}
\begin{empheq}[left={}\empheqlbrace]{alignat=2}
r_{k}&=\frac{1}{\eta_k}\sqrt{\mathsf{R}(\mathbf{u}^{k})},\\[0.1cm]
s_{k}&=\frac{\mu\mathsf{R}(\bar{\mathbf{u}}^{k})}{\eta_k\sqrt{\mathsf{R}(\mathbf{u}^{k})}}.
\end{empheq}
\end{subequations}
To get a non-increasing sequence $\{r_{k},s_{k}\}$, the scaling factor $\eta_k$ should satisfy
\begin{equation}
\eta_k \ge \max \left\{\eta_{k-1} \cdot \sqrt{\frac{\mathsf{R}(\mathbf{u}^{k})}{\mathsf{R}(\mathbf{u}^{k-1})}},\quad \eta_{k-1} \cdot \frac{\mathsf{R}(\bar{\mathbf{u}}^{k}) \sqrt{\mathsf{R}(\mathbf{u}^{k-1})}}{\mathsf{R}(\bar{\mathbf{u}}^{k-1}) \sqrt{\mathsf{R}(\mathbf{u}^{k})}} \right\}.\label{sec4-add13}
\end{equation}
For any iteration $k\ge0$, the difference matrix $\mathbf{D}_{k}=\mathbf{H}_{k}-\mathbf{H}_{k+1}$ is semi positive definite. The initial extended matrix can be written as $\mathbf{H}_{0} = \frac{1}{\eta_0}\mathsf{H}_{0}$, where 
\begin{eqnarray} 
  \mathsf{H}_{0}  =  \left(\begin{array}{ccc}
       \sqrt{\mathsf{R}(\mathbf{u}^{0})}\mathbf{I}_{n} &  \mathbf{0}& \mathbf{0}\\[0.1cm]
       \mathbf{0}&  \sqrt{\mathsf{R}(\mathbf{u}^{0})}\mathbf{I}_{m} &  \mathbf{0}\\[0.1cm]
        \mathbf{0}& \mathbf{0}   & \frac{\mu\mathsf{R}(\bar{\mathbf{u}}^{0})}{\sqrt{\mathsf{R}(\mathbf{u}^{0})}} \mathbf{I}_{p} \end{array} \right).\nonumber
\end{eqnarray}

\subsection{Convergence Analysis}
\begin{lemma}\label{sec4-lemma2}
Let $\{\mathbf{w}^{k},\bar{\mathbf{w}}^{k},\mathbf{w}^{k+1}\}$ be the sequences generated by the \textsf{Spice} method. When we use equation (\ref{sec4-eq12}) to update the values of $r_{k}$ and $s_{k}$, the predictive matrix $\mathbf{Q}_{k}$ and corrective matrix $\mathbf{M}_{k}$ are upper triangular and satisfy
\begin{equation}
\mathbf{H}_{k}=\mathbf{Q}_{k}\mathbf{M}_{k}^{-1}\succ0 \quad {\rm and}\quad  \mathbf{G}_{k}=\mathbf{Q}_{k}^{\top}+\mathbf{Q}_{k}-\mathbf{M}_{k}^{\top}\mathbf{H}_{k}\mathbf{M}_{k}\succ0,\label{sec4-eq13}
\end{equation}
 then the following variational inequality holds
\begin{equation}
  \rho[\vartheta(\mathbf{u})- \vartheta(\bar{\mathbf{u}}^{k})]+(\mathbf{w}-\bar{\mathbf{w}}^{k})^{\top}\frac{1}{\eta_{k}}\boldsymbol{\Gamma}(\bar{\mathbf{w}}^{k})\ge\frac{1}{2}\left(\|\mathbf{w}-\mathbf{w}^{k+1}\|_{\mathbf{H}_{k}}^{2}-\|\mathbf{w}-\mathbf{w}^{k}\|_{\mathbf{H}_{k}}^{2}\right) +\frac{1}{2}\|\bar{\mathbf{w}}^{k}-\mathbf{w}^{k}\|_{\mathbf{G}_{k}}^{2},\quad \forall \ \mathbf{w}\in \Omega. \label{sec4-eq14}
\end{equation}
\end{lemma}
\begin{proof}
The proof is omitted since it is similar to that of Lemma \ref{sec3-lemma2}.
\end{proof}

\begin{theorem}
Let $\{\mathbf{w}^{k},\bar{\mathbf{w}}^{k},\mathbf{w}^{k+1}\}$ be the sequences generated by the \textsf{Spice} method. For the predictive matrix $\mathbf{Q}_{k}$, if there is a corrective matrix $\mathbf{M}_{k}$ that satisfies the convergence condition (\ref{sec4-eq14}), then we have
\begin{equation}
\Vert\mathbf{w}^{*}-\mathbf{w}^{k}\Vert_{\mathbf{H}_{k}}^{2}\ge\Vert\mathbf{w}^{*}-\mathbf{w}^{k+1}\Vert_{\mathbf{H}_{k}}^{2}+\Vert\mathbf{w}^{k}-\bar{\mathbf{w}}^{k}\Vert_{\mathbf{G}_{k}}^{2}, \quad \ \mathbf{w}^{*}\in \Omega^{*}, \label{sec4-eq15}
\end{equation}
where $\Omega^{*}$ is the set of optimal solutions.
\end{theorem}
\begin{proof}
The proof is omitted since it is similar to that of Theorem \ref{sec3-theorem1}.
\end{proof}

\begin{theorem}
The $\{\mathbf{w}^k,\bar{\mathbf{w}}^k\}$ the sequences generated by the \textsf{Spice} method. For the given optimal solution $\mathbf{w}^{*}$, the following limit equations hold
\begin{equation}
\lim_{k \to \infty} \|\mathbf{w}^k - \bar{\mathbf{w}}^k\|^2 = 0 \quad {\rm and}\quad \lim_{k \to \infty}\mathbf{w}^{k}=\mathbf{w}^{*}.\label{sec3-eq12}
\end{equation}
\end{theorem}
\begin{proof}
The proof is omitted since it is similar to that of Theorem \ref{sec3-theorem2}.
\end{proof}

\begin{theorem}
Let $\{\bar{\mathbf{x}}^{k}\}$ be the sequence generated by  the \textsf{Spice} method for \textsf{P2} and $\eta_{k}$ satisfy condition (\ref{sec4-add13}). Let $\bar{\mathbf{u}}_{t}$, $\eta_{t}$, and $\bar{\mathbf{w}}_{t}$ be defined as follows
\begin{equation}
\bar{\mathbf{u}}_{t}=\frac{1}{t+1}\sum_{k=0}^{t}\bar{\mathbf{u}}^{k},\quad \eta_{t}=\frac{t+1}{\sum_{k=0}^{t}\frac{1}{\eta_{k}}}, \quad \bar{\mathbf{w}}_{t}=\frac{\sum_{k=0}^{t}\frac{\bar{\mathbf{w}}^{k}}{\eta_{k}}}{\sum_{k=0}^{t}\frac{1}{\eta_{k}}}.\nonumber
\end{equation}
If Lemma \ref{sec4-lemma2} holds, for any iteration $t>0$ and scaling factor $\rho(t)>0$, the error bound satisfies
\begin{equation}
\mathbf{w}^{*}\in \Omega,\quad \vartheta(\bar{\mathbf{u}}_{t})-\vartheta(\mathbf{u}^{*})+(\bar{\mathbf{w}}_{t}-\mathbf{w}^{*})^{\top}\frac{1}{\eta_{t}\rho(t)}\boldsymbol{\Gamma}(\mathbf{w}^{*})\leq \frac{1}{2\eta_0 \rho(t)(t+1)}\Vert\mathbf{w}^{*}-\mathbf{w}^{0}\Vert_{\mathsf{H}_{0}}^{2}, \quad\forall \ \bar{\mathbf{w}}_{t}\in \Omega. \label{sec4-eq17}
\end{equation}
\end{theorem}
\begin{proof}
The proof is omitted since it is similar to that of Theorem \ref{sec3-theorem3}.
\end{proof}

\begin{corollary}
The \textsf{Spice} method provides flexibility in controlling the convergence rate through the choice of the scaling factor $\rho(t)$. For example, setting $\rho(t) = (t+1)^{\alpha}$ with $\alpha > 0$ yields a power-law convergence rate, while $\rho(t) = e^{\beta t}$ with $\beta > 0$ achieves exponential convergence. Additionally, $\rho(t) = (t+1)^{t}$ combines both power-law and exponential effects. This adaptability ensures that the method can tailor its convergence behavior to different nonlinear convex optimization problems, converging asymptotically as $\rho(t)$ increases.
\end{corollary}

\subsection{Solving Convex Problems with Equality Constraints}
The \textsf{Spice} method can be extended to address convex problems involving nonlinear inequality and linear equation constraints. The convex problem is formulated as follows:
\begin{equation}
\begin{aligned}
\min\left\{ f(\mathbf{x})+g(\mathbf{y}) \ \middle|\   \phi_{i}(\mathbf{x}) +\psi_{i}(\mathbf{y})\leq 0,\ \mathbf{A}\mathbf{x}+\mathbf{B}\mathbf{y}=\mathbf{b},\ \mathbf{x}\in \mathcal{X},\ \mathbf{y}\in \mathcal{Y}, \ i=1,\cdots,p_{1}\right\},
\end{aligned}\nonumber
\end{equation}
where $\mathbf{A}\in \mathbb{R}^{p_{2}\times n}$, $\mathbf{B}\in \mathbb{R}^{ p_{2}\times m}$, and $\mathbf{b}\in \mathbb{R}^{ p_{2}}$ ($p_{1}+p_{2}=:p$).
 The constraint functions are structured as  $\Phi(\mathbf{x})=[\phi_{1}(\mathbf{x}),\ldots,\phi_{p_{1}}(\mathbf{x}), \phi_{p_{1}+1}(\mathbf{x}),\cdots,\phi_{p}(\mathbf{x})]^{\top}$, with $[\phi_{p_{1}+1}(\mathbf{x}),\cdots,\phi_{p}(\mathbf{x})]^{\top}=\mathbf{A}\mathbf{x}-\mathbf{b}/2$, and $\Psi(\mathbf{y})=[\psi_{1}(\mathbf{y}),\cdots,\psi_{p_{1}}(\mathbf{y}),\psi_{p_{1}+1}(\mathbf{y}),\cdots,\psi_{p}(\mathbf{y})]^{\top}$, with $[\psi_{p_{1}+1}(\mathbf{y}),\cdots,\psi_{p}(\mathbf{y})]^{\top}=\mathbf{B}\mathbf{y}-\mathbf{b}/2$. The domain of the dual variable $\boldsymbol{\lambda}$ must be redefined to account for these constraints. Specifically, the domain set of $\boldsymbol{\lambda}$ becomes $\mathcal{Z} := \mathbb{R}^{p_{1}}_{+} \times \mathbb{R}^{p_{2}}$, where $\mathbb{R}^{p_{1}}_{+}$ handles the non-positivity constraints and $\mathbb{R}^{p_{2}}$ corresponds to the equality constraints.

\section{Numerical Experiment}

To evaluate the performance of the \textsf{Spice} method, this paper examines quadratically constrained quadratic programming (QCQP) problems that commonly arise in control theory and signal processing. The QCQP structure allows for closed-form solutions at each iteration. Single-variable and separable-variable QCQP problems involving nonlinear objective and constraint functions are considered. This approach ensures a comprehensive evaluation of the method’s performance across different scenarios.

 \subsection{Single-Variable QCQP}
In single-variable QCQP, the objective and constraints are quadratic functions, which can be described as:
\begin{align}
\min_{\mathbf{x}} \quad f(\mathbf{x})=&\|\mathbf{W}_0\mathbf{x} - \mathbf{a}_0\|^2  \nonumber\\
\text{subject to} \quad & \|\mathbf{W}_i \mathbf{x} - \mathbf{a}_i\|^2 \leq \pi_{i},\quad i =1,\cdots,p, \nonumber
\end{align}
 where $\mathbf{x}\in \mathbb{R}^{n}$ and $\mathbf{W}_{i}\in \mathbb{R}^{q\times n}$ and $\mathbf{a}_{i}\in \mathbb{R}^{q}, i=0,1,\cdots,p$.  
\subsubsection{PPA-Like Prediction Scheme}
In the first step, the PPA-like prediction scheme updates primal and dual variables by solving the corresponding subproblems. 
\begin{subequations}
\begin{empheq}[left=\empheqlbrace]{alignat=2}
 \bar{\mathbf{x}}^{k} &=\arg\min \big\{ \mathcal{L}(\mathbf{x},\boldsymbol{\lambda}^{k},\rho,\eta_{k}) + \frac{r_{k}}{2}\|\mathbf{x}-\mathbf{x}^k\|^2 \mid \mathbf{x}\in \mathcal{X} \big\},\nonumber\\[0.1cm]
     \boldsymbol{\bar{\lambda}}^{k} &= \arg\max \big\{ \mathcal{L}(\bar{\mathbf{x}}^{k}, \boldsymbol{\lambda},\rho,\eta_{k}) - \frac{s_{k}}{2}\|\boldsymbol{\lambda}-\boldsymbol{\lambda}^k\|^2 \mid  \boldsymbol{\lambda}\in \mathcal{Z} \big\},\nonumber
\end{empheq}
\end{subequations}
By solving the above subproblems, we obtain
\begin{empheq}[]{alignat=2}
&\bar{\mathbf{x}}^{k}=\left(2\rho\mathbf{W}_{0}^{\top}\mathbf{W}_{0}+\frac{2}{\eta_{k}}\sum_{i=1}^{p} \lambda^{k}_{i}\mathbf{W}^{\top}_{i}\mathbf{W}_{i}+r_{k}\mathbf{I}_{n}\right)^{-1}\cdot\left(2\rho\mathbf{W}_{0}^{\top}\mathbf{a}_{0}+\frac{2}{\eta_{k}}\sum_{i=1}^{p} \lambda^{k}_{i}\mathbf{W}^{\top}_{i}\mathbf{a}_{i}+r_{k}\mathbf{x}^{k}\right),\nonumber\\
&\bar{\lambda}^{k}_{i}=\max\{\hat{\lambda}^{k}_{i},0\}, \mbox{ where } \hat{\lambda}^{k}_{i}= \lambda^{k}_{i}+\frac{1}{\eta_{k} s_{k}}\left(\Vert\mathbf{W}_{i}\bar{\mathbf{x}}^{k}-\mathbf{a}_{i}\Vert^{2}- \pi_{i}\right), \quad i=1,\cdots,p.\nonumber
\end{empheq}
\subsubsection{Matrix-Driven Correction Scheme}
In the second step, the prediction error is corrected by an upper triangular matrix. The iterative scheme of the variables is given by:
\begin{equation}
\mathbf{w}^{k+1} =\mathbf{w}^{k}- \mathbf{M}_{k}(\mathbf{w}^{k}-\bar{\mathbf{w}}^{k}),\nonumber
\end{equation}
where the corrective matrix $\mathbf{M}_{k}$ is given as 
\begin{equation}
\mathbf{M}_{k}=\left( \begin{array}{cc}\mathbf{I}_{n} &-\frac{1}{\eta_k r_{k}}\mathcal{D} \Phi (\bar{\mathbf{x}}^{k})^{\top}\\[0.1cm]
        \mathbf{0} & \mathbf{I}_{p}\end{array} \right),\nonumber 
\end{equation}
where $\mathcal{D} \Phi (\mathbf{x})=[\nabla \phi_{1},\cdots,\nabla \phi_{p}]^{\top}\in \mathbb{R}^{p\times n}$. Let $\mathsf{R}(\mathbf{x}^{k})=\|\mathcal{D} \Phi (\mathbf{x}^{k})\|^{2}$ and the regularization parameters $r_{k}$ and $s_{k}$ are defined as follows: 
\begin{equation}
r_{k}=\frac{1}{\eta_k}\sqrt{\mathsf{R}(\mathbf{x}^{k})},\quad
s_{k}=\frac{\mu\mathsf{R}(\bar{\mathbf{x}}^{k})}{\eta_k\sqrt{\mathsf{R}(\mathbf{x}^{k})}}.\nonumber
\end{equation}
To ensure that $r_{k-1} \geq r_{k}$ and $s_{k-1} \geq s_{k}$, the scaling factor $\eta_k$ must be satisfied the following condition. 
\begin{equation}
\eta_k \ge \max \left\{\eta_{k-1} \cdot \sqrt{\frac{\mathsf{R}(\mathbf{x}^{k})}{\mathsf{R}(\mathbf{x}^{k-1})}},\quad \eta_{k-1} \cdot \frac{\mathsf{R}(\bar{\mathbf{x}}^{k}) \sqrt{\mathsf{R}(\mathbf{x}^{k-1})}}{\mathsf{R}(\bar{\mathbf{x}}^{k-1}) \sqrt{\mathsf{R}(\mathbf{x}^{k})}} \right\}.\nonumber
\end{equation}
 
  \subsection{Separable-Variable QCQP}

For the separable-variable QCQP, the problem structure is extended to consider two sets of variables, $\mathbf{x}$ and $\mathbf{y}$, which are optimized independently within a coupled quadratic framework. The optimization problem is formulated as follows:
\begin{align}
\min_{\mathbf{x},\mathbf{y}} \quad f(\mathbf{x})=&\|\mathbf{W}_0\mathbf{x} - \mathbf{a}_0\|^2 +\|\mathbf{V}_0\mathbf{y} - \mathbf{c}_0\|^2 \nonumber\\
\text{subject to} \quad & \|\mathbf{W}_i \mathbf{x} - \mathbf{a}_i\|^2+\|\mathbf{V}_i \mathbf{y} - \mathbf{c}_i\|^2 \leq \pi_{i},\quad i =1,\cdots,p,\nonumber
\end{align}
where $\mathbf{x}\in \mathbb{R}^{n}$, $\mathbf{y}\in \mathbb{R}^{m}$ and $\mathbf{W}_{i}\in \mathbb{R}^{q\times n}$,$\mathbf{V}_{i}\in \mathbb{R}^{q\times m}$. $\mathbf{a}_{i}\in \mathbb{R}^{q}$ and $\mathbf{c}_{i}\in \mathbb{R}^{q}, i=0,1,\cdots,q$.

\subsubsection{PPA-Like Prediciton Scheme}
In the prediction step, we update the primal and dual variables by solving the following subproblems. 

\begin{subequations}
\begin{empheq}[left=\empheqlbrace]{alignat=2} 
 \bar{\mathbf{x}}^{k} &=\arg\min \big\{ \mathcal{L}(\mathbf{x},\mathbf{y}^{k},\boldsymbol{\lambda}^{k},\rho,\eta_{k}) + \frac{r_{k}}{2}\|\mathbf{x}-\mathbf{x}^k\|^2 \mid \mathbf{x}\in \mathcal{X} \big\},\nonumber\\[0.1cm]
  \bar{\mathbf{y}}^{k} &=\arg\min \big\{ \mathcal{L}(\bar{\mathbf{x}}^{k},\mathbf{y},\boldsymbol{\lambda}^{k},\rho,\eta_{k}) + \frac{r_{k}}{2}\|\mathbf{y}-\mathbf{y}^k\|^2 \mid \mathbf{y}\in \mathcal{Y} \big\},\nonumber\\[0.1cm]
     \boldsymbol{\bar{\lambda}}^{k}&= \arg\max \big\{ \mathcal{L}(\bar{\mathbf{x}}^{k}, \bar{\mathbf{y}}^{k} ,\boldsymbol{\lambda},\rho,\eta_{k}) - \frac{s_{k}}{2}\|\boldsymbol{\lambda}-\boldsymbol{\lambda}^k\|^2 \mid  \boldsymbol{\lambda}\in \mathcal{ Z} \big\},\nonumber
\end{empheq}
\end{subequations}
By solving the above subproblems, we obtain
\begin{empheq}[]{alignat=2}
&\bar{\mathbf{x}}^{k}=\left(2\rho\mathbf{W}_{0}^{\top}\mathbf{W}_{0}+\frac{2}{\eta_{k}}\sum_{i=1}^{p} \lambda^{k}_{i}\mathbf{W}^{\top}_{i}\mathbf{W}_{i}+r_{k}\mathbf{I}_{n}\right)^{-1}\cdot\left(2\rho\mathbf{W}_{0}^{\top}\mathbf{a}_{0}+\frac{2}{\eta_{k}}\sum_{i=1}^{p} \lambda^{k}_{i}\mathbf{W}^{\top}_{i}\mathbf{a}_{i}+r_{k}\mathbf{x}^{k}\right),\nonumber\\
&\bar{\mathbf{y}}^{k}=\left(2\rho\mathbf{V}_{0}^{\top}\mathbf{V}_{0}+\frac{2}{\eta_{k}}\sum_{i=1}^{p} \lambda^{k}_{i}\mathbf{V}^{\top}_{i}\mathbf{V}_{i}+r_{k}\mathbf{I}_{n}\right)^{-1}\cdot\left(2\rho\mathbf{V}_{0}^{\top}\mathbf{c}_{0}+\frac{2}{\eta_{k}}\sum_{i=1}^{p} \lambda^{k}_{i}\mathbf{V}^{\top}_{i}\mathbf{c}_{i}+r_{k}\mathbf{y}^{k}\right),\nonumber\\
&\bar{\lambda}^{k}_{i}=\max\{\hat{\lambda}^{k}_{i},0\}, \mbox{ where } \hat{\lambda}^{k}_{i}= \lambda^{k}_{i}+\frac{1}{\eta_{k} s_{k}}\left(\Vert\mathbf{W}_{i}\bar{\mathbf{x}}^{k}-\mathbf{a}_{i}\Vert^{2}+\Vert\mathbf{V}_{i}\bar{\mathbf{y}}^{k}-\mathbf{c}_{i}\Vert^{2}- \pi_{i}\right), \quad i=1,\cdots,p.\nonumber
\end{empheq}
\subsubsection{Matrix-Driven Correction Scheme}
In the correction step, the prediction error is corrected by using an upper triangular matrix. The update rule is:
\begin{equation}
\mathbf{w}^{k+1} =\mathbf{w}^{k}- \mathbf{M}_{k}(\mathbf{w}^{k}-\bar{\mathbf{w}}^{k}),\nonumber
\end{equation}
where the corrective matrix $\mathbf{M}_{k}$ is given as 
\begin{equation} \mathbf{M}_{k}=\left(
\begin{array}{ccc}
\mathbf{I}_{n} &\mathbf{0}& -\frac{1}{\eta_k r_{k}}\mathcal{D} \Phi (\bar{\mathbf{x}}^{k})^{\top} \\[0.2cm]\mathbf{0}&\mathbf{I}_{m}& -\frac{1}{\eta_k r_{k}}\mathcal{D} \Psi (\bar{\mathbf{y}}^{k})^{\top}\\[0.2cm]
\mathbf{0}&\mathbf{0}&\mathbf{I}_{p}
\end{array}\right),\nonumber
\end{equation}
where $\mathcal{D} \Phi (\mathbf{x})=[\nabla \phi_{1},\cdots,\nabla \phi_{p}]^{\top}\in \mathbb{R}^{p\times n}$ and $\mathcal{D} \Psi (\mathbf{y})=[\nabla \psi_{1},\cdots,\nabla \psi_{p}]^{\top}\in \mathbb{R}^{p\times m}$. For $\mu>1$ and $k\ge1$, we set $\mathsf{R}(\mathbf{u}^{k})=\|\mathcal{D} \Phi (\mathbf{x}^{k})\|^{2}+\|\mathcal{D} \Psi (\mathbf{y}^{k})\|^{2}$ and take a non-decreasing sequence like
\begin{equation}
r_{k}=\frac{1}{\eta_k}\sqrt{\mathsf{R}(\mathbf{u}^{k})},\quad
s_{k}=\frac{\mu\mathsf{R}(\bar{\mathbf{u}}^{k})}{\eta_k\sqrt{\mathsf{R}(\mathbf{u}^{k})}}.\nonumber
\end{equation}
To get a non-increasing sequence $\{r_{k},s_{k}\}$, the scaling factor $\eta_k$ should satisfy
\begin{equation}
\eta_k \ge \max \left\{\eta_{k-1} \cdot \sqrt{\frac{\mathsf{R}(\mathbf{u}^{k})}{\mathsf{R}(\mathbf{u}^{k-1})}},\quad \eta_{k-1} \cdot \frac{\mathsf{R}(\bar{\mathbf{u}}^{k}) \sqrt{\mathsf{R}(\mathbf{u}^{k-1})}}{\mathsf{R}(\bar{\mathbf{u}}^{k-1}) \sqrt{\mathsf{R}(\mathbf{u}^{k})}} \right\}.\nonumber
\end{equation}
 
 \subsection{Parameter Setting}

The dimensionality of the vectors $ \mathbf{x} $ and $ \mathbf{y} $ is set to $ n = 300 $ and $ m = 300 $, respectively, while the dimension $ q $ for the matrices $ \mathbf{W}_i $, $ \mathbf{V}_i $, and the vectors $ \mathbf{a}_i $, $ \mathbf{c}_i $ is set to 400. The number of constraints $ p $ is set to 20. To ensure reproducibility, a random seed of 0 is used. Matrices $ \mathbf{W}_0 \in \mathbb{R}^{q \times n} $ and $ \mathbf{V}_0 \in \mathbb{R}^{q \times m} $ are generated by drawing elements from a standard normal distribution and scaling them by 1. Vectors $ \mathbf{a}_0 \in \mathbb{R}^q $ and $ \mathbf{c}_0 \in \mathbb{R}^q $ are drawn from the same distribution and scaled by 12. For each constraint, $ i = 1, \dots, p $, the matrices $ \mathbf{W}_{i} \in \mathbb{R}^{q \times n} $ and $ \mathbf{V}_{i} \in \mathbb{R}^{q \times m} $ are generated by drawing elements from a standard normal distribution. The corresponding vectors $ \mathbf{a}_i \in \mathbb{R}^q $ and $ \mathbf{c}_i \in \mathbb{R}^q $ are scaled by 0.1. The bounds of two problems for the constraints $ \pi_i $ are uniformly set to 500,000 and 1,000,000, respectively. The tolerant error is set to $10^{-9}$. The values of $\alpha$ and $\beta$ are set to 2. In addition, the delta objective value is defined as the iterative error $\textsf{abs}(f(\mathbf{x}^{k})-f(\mathbf{x}^{k+1}))$.

 \begin{figure}[t]
\centering
\subfigure[Single-variable QCQP: objective function value ]{
\includegraphics[width=7.5cm]{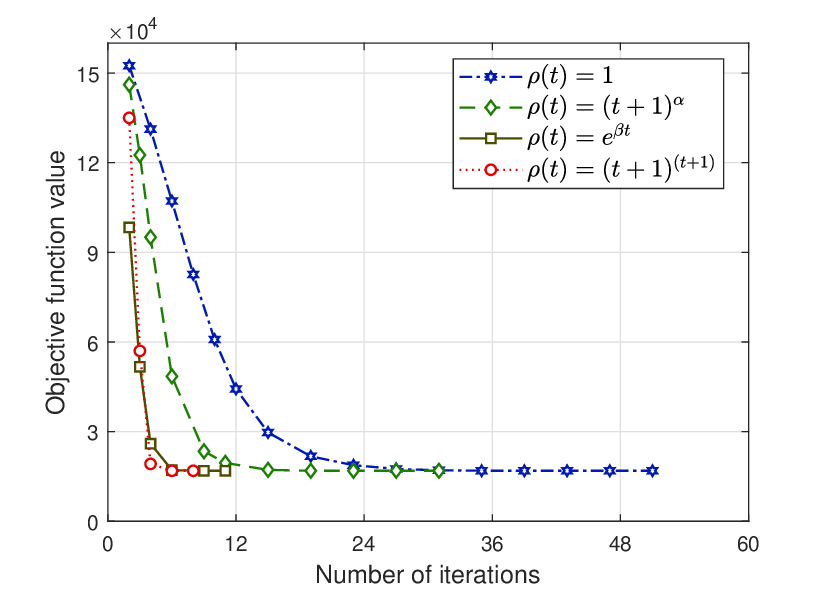}
}\hspace{5mm}
\subfigure[Single-variable QCQP: delta objective value]{
\includegraphics[width=7.5cm]{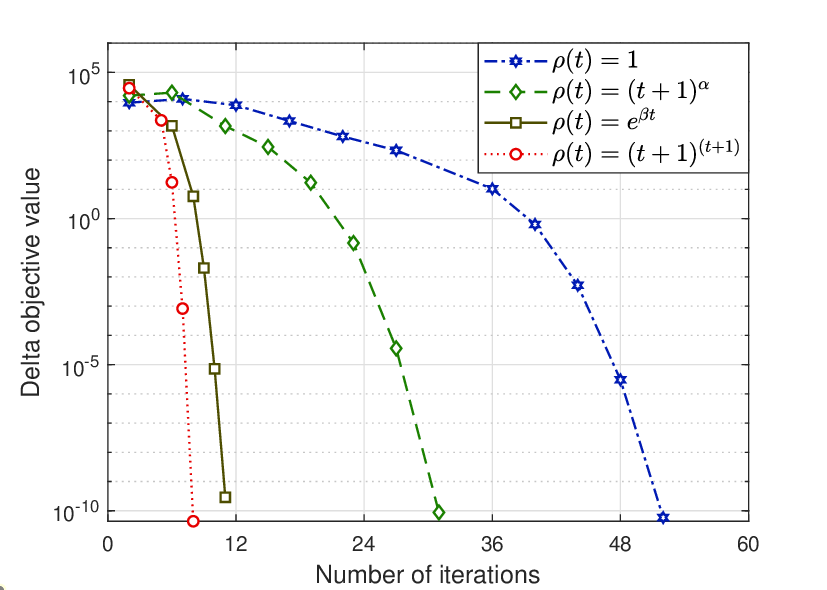}
}
\caption{Performance of the \textsf{Spice} method for single-variable QCQP problems with different $\rho(t)$ functions. (a) Objective function value versus the number of iterations; (b) delta objective value versus the number of iterations.}
\label{fig1}
\end{figure}

 \begin{figure}[t]
\centering
\subfigure[Separable-variable QCQP: objective function value ]{
\includegraphics[width=7.5cm]{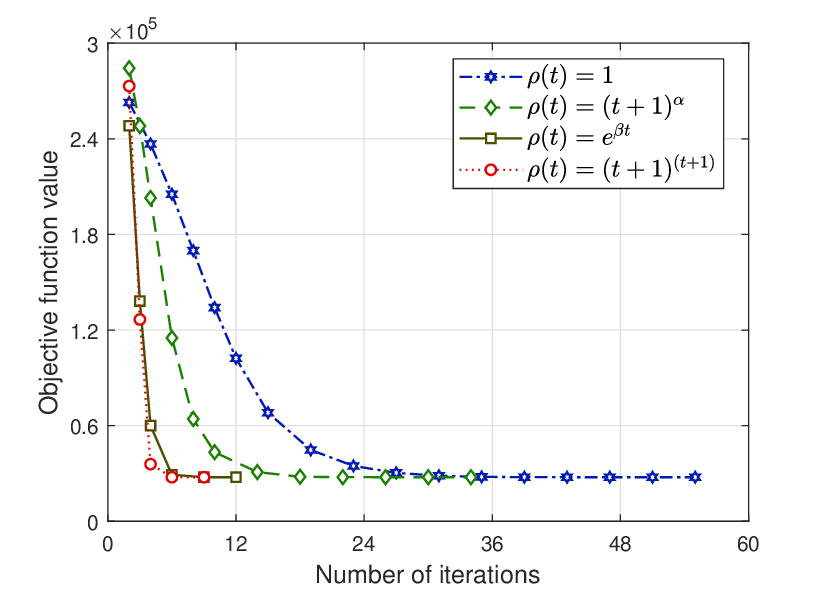}
}\hspace{5mm}
\subfigure[Separable-variable QCQP: delta objective value]{
\includegraphics[width=7.5cm]{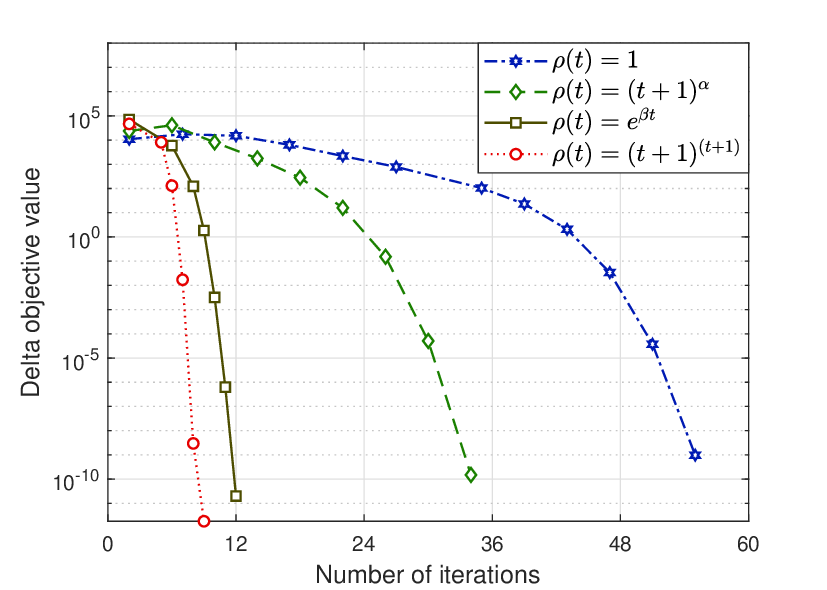}
}
\caption{Performance of the \textsf{Spice} method for separable-variable QCQP problems with different $\rho(t)$ functions. (a) Objective function value versus the number of iterations; (b) delta objective value versus the number of iterations.}
\label{fig2}
\end{figure}
 
 \subsection{Results Analysis}
 
 In Figure \ref{fig1}, the performance of the \textsf{Spice} method for single-variable QCQP problems is evaluated. In subplot (a),  we show the objective function value versus the number of iterations. The different curve legends represent four different settings of the $\rho(t)$ function: $\rho(t) = 1$, $\rho(t) = (t+1)^\alpha$, $\rho(t) = e^{\beta t}$, and $\rho(t) = (t+1)^{t+1}$. The curves show a steady decrease in the objective function values as the number of iterations increases, with $\rho(t) = (t+1)^{t+1}$ converging faster than the others, particularly within the first 10 iterations. The other settings for $\rho(t)$ demonstrate a slower but consistent convergence behavior. Subplot (b) indicates the delta objective value, plotted on a logarithmic scale, which provides more insight into the convergence rate. Again, $\rho(t) = (t+1)^{t+1}$ demonstrates a quicker decline in the delta objective value, indicating faster convergence. The results show that the choice of $\rho(t)$ significantly impacts the convergence rate, with more aggressive functions leading to faster reductions in both the objective and delta objective values.

In Figure \ref{fig2}, the \textsf{Spice} method's performance is extended to separable-variable QCQP problems. Subplot (a) shows the objective function values across iterations, where the same $\rho(t)$ functions are used. As expected, $\rho(t) = (t+1)^{t+1}$ exhibits the fastest convergence, closely followed by $\rho(t) = e^{\beta t}$, while the constant function $\rho(t) = 1$ leads to a much slower decrease in the objective value. Subplot (b) shows the delta objective value on a logarithmic scale, revealing a similar trend where $\rho(t) = (t+1)^{t+1}$ outperforms the others in reducing the objective error more rapidly. The overall conclusion from these figures is that the adaptive choices for $\rho(t)$ can significantly accelerate convergence, with functions that grow with iterations providing more efficient performance in both single-variable and separable-variable QCQP problems.

\setlength{\aboverulesep}{0pt}
\setlength{\belowrulesep}{0pt}
\begin{table*}[thb]\centering
\caption{Number of iterations: PC and \textsf{Spice} methods}
\scalebox{0.9}{
\begin{tabular}{@{}ccccccccccccc@{}}\toprule
\multicolumn{3}{c}{\textbf{Dimensions}} & \phantom{.}& \multicolumn{4}{c}{\textbf{Single-variable QCQP}} & \phantom{.} & \multicolumn{4}{c}{\textbf{ Separable-variable QCQP} } \\ 
\cmidrule{1-3} \cmidrule{5-8} \cmidrule{10-13}
$n$ & $m$ & p &   & PC & &\textsf{Spice}: $\rho(t)=1$ & \textsf{Spice}: $\rho(t)=e^{\beta t}$  &   & PC & &\textsf{Spice}: $\rho(t)=1$ & \textsf{Spice}: $\rho(t)=e^{\beta t}$  \\ \midrule
100 & 100 & 10 & & 546 & & 31 & 8 &  & 836 & & 33 & 8  \\
300 & 300 & 10 &    & 12462 & & 49 & 10 &   & 24068 &  & 53 & 10  \\ \midrule
100 & 100 & 20 &    & 745 &  & 32 & 8 &  & 1129 &  & 35& 9  \\
300 & 300 & 20 &   & 16890 &  & 51 & 10 &  & 33206 & & 54 & 11 \\
\bottomrule
\end{tabular}}
\label{table1}
\end{table*}

Table \ref{table1} compares the number of iterations required by the PC method and the \textsf{Spice} method for both single-variable and separable-variable QCQP problems across different problem dimensions. The dimensions $n$, $m$, and $p$ refer to the problem size, with larger values representing more complex instances. For single-variable QCQP problems, the PC method consistently requires a significantly larger number of iterations than the \textsf{Spice} method, regardless of the $\rho(t)$ function used. When $\rho(t) = 1$, the number of iterations needed by the \textsf{Spice} method is greatly reduced, and when $\rho(t) = e^{\beta t}$, the iteration count decreases even further. For instance, in the case where $n = 100$, $m = 100$, and $p = 10$, the PC method requires 546 iterations, while the \textsf{Spice} method with $\rho(t) = e^{\beta t}$ only takes 8 iterations, demonstrating the significant improvement in efficiency. Similarly, for separable-variable QCQP problems, the \textsf{Spice} method outperforms the PC method regarding the number of iterations. The trend persists as the problem dimensions increase, with the \textsf{Spice} method maintaining a more efficient convergence pattern, particularly with the $\rho(t) = e^{\beta t}$ function. For example, in the case of $n = 300$, $m = 300$, and $p = 20$, the PC method requires 33,206 iterations, whereas the \textsf{Spice} method with $\rho(t) = e^{\beta t}$ converges in just 11 iterations. This highlights the dramatic improvement in scalability offered by the \textsf{Spice} method with adaptive $\rho(t)$. This table illustrates the superior convergence performance of the \textsf{Spice} method, especially when utilizing the exponential $\rho(t)$ function, which achieves faster convergence and reduces the number of iterations by several orders of magnitude compared to the PC method.

\section{Conclusion}

This paper introduced a novel scaling technique to adjust the weights of the objective and constraint functions. Based on this technique, the \textsf{Spice} method was designed to achieve a free convergence rate, addressing limitations inherent in traditional PC approaches. Additionally, the \textsf{Spice} method was extended to handle separable-variable nonlinear convex problems. Theoretical analysis, supported by numerical experiments, demonstrated that varying the scaling factors for the objective and constraint functions results in flexible convergence rates. These findings underscore the practical efficacy of the \textsf{Spice} method, providing a robust framework for solving a wider range of nonlinear convex optimization problems.

\bibliographystyle{elsarticle-num}
\bibliography{newref}

\end{document}